\documentclass[journal,twoside,final]{IEEEtran}

\usepackage{graphicx}
\usepackage{amsmath,amssymb,amsfonts,amsthm}
\usepackage[scr=dutchcal]{mathalpha}
\usepackage{grffile}
\usepackage{soul}
\usepackage{multicol}
\usepackage{caption}
\usepackage{subcaption}
\usepackage{url}
\usepackage{stfloats,titlesec}
\usepackage{enumitem}
\usepackage{algorithm,algpseudocode}
\usepackage{xcolor}

\renewcommand{\thesubsubsection}{\thesubsection.\arabic{subsubsection}}
\titlespacing*{\subsection}{0pt}{0.9\baselineskip}{0.6\baselineskip}

\let\bbbbl\Biggl

\let\bbbbr\Biggr

\newcommand{\mcl}[1]{\mathcal{ #1}}
\newcommand{\mbf}[1]{\mathbf{ #1}}
\newcommand{\norm}[1]{\left\Vert #1\right\Vert}
\newcommand{\hinf}{\ensuremath{H_{\infty}}}
\newcommand{\ip}[2]{\left\langle{#1},{#2}\right\rangle}
\newcommand{\half}{\frac{1}{2}}

\newcommand{\bmat}[1]{\begin{bmatrix} #1\end{bmatrix}}

\newcommand{\mat}[1]{\begin{matrix}#1\end{matrix}}

\newcommand{\R}{\mathbb{R}}
\newcommand{\C}{\mathbb{C}}

\newcommand{\N}{\mathbb{N}}
\newcommand{\Z}{\mathbb{Z}}

\newtheorem{thm}{Theorem}
\newtheorem{defn}[thm]{Definition}
\newtheorem{lem}[thm]{Lemma}

\newtheorem{cor}[thm]{Corollary}
\newtheorem{rem}{Remark}

\newtheorem{ex}[thm]{\textbf{Example}}

\newcommand{\PI}{\pmb{\Pi}}
\newcommand{\pie}{\scalebox{0.9}[1.2]{$\mathit{\Pi}$}}

\newcommand{\fourpi}[4]{\hspace{0.5mm}\pie\hspace{-1.mm}\left[\footnotesize\begin{array}{c|c}
#1&#2\\\hline #3 & \{#4\}
\end{array}\right]}
\newcommand{\threepi}[1]{\pie_{\{#1\}}}

\makeatletter
\newenvironment{breakablealgorithm}
  {
  \begin{center}
     \refstepcounter{algorithm}
     \hrule height.8pt depth0pt \kern2pt
     \renewcommand{\caption}[2][\relax]{
       {\raggedright\textbf{\fname@algorithm~\thealgorithm} ##2\par}%
       \ifx\relax##1\relax 
         \addcontentsline{loa}{algorithm}{\protect\numberline{\thealgorithm}##2}%
       \else 
         \addcontentsline{loa}{algorithm}{\protect\numberline{\thealgorithm}##1}%
         \fi
         \kern2pt\hrule\kern2pt
         }
  }{
     \kern2pt\hrule\relax
     \end{center}
     }
     \makeatother

     \allowdisplaybreaks

     \title{
     Dual Representations and $\hinf$-Optimal Control of Partial Differential Equations
     }

\author{Sachin~Shivakumar,~Amritam~Das,~and~Matthew~Peet,~\IEEEmembership{Senior Member,~IEEE}
\thanks{S. Shivakumar is with CNLS and T-5 Division, Los Alamos National Laboratory  (email: sshivakumar@lanl.gov).}%
\thanks{M. Peet is with School for Engineering of Matter, Transport and Energy, Arizona State University (email: mpeet@asu.edu).}
\thanks{A. Das is with Department of Electrical Engineering, Eindhoven University of Technology (email: am.das@tue.nl).}
\thanks{This work was supported by National Science Foundation under Grants No. 2337751, 2429973, CMMI-1935453 and CNS-1739990.}
\thanks{This work was partly supported by LDRD program of Los Alamos National Laboratory. LA-UR-25-31030}
}

\newcommand{\blue}[1]{\color{black} #1 \color{black}}

\begin{document}

\maketitle

\begin{abstract}                
We consider $\hinf$-optimal state-feedback control of the class of linear Partial Differential Equations (PDEs) which admit a Partial Integral Equation (PIE) representation. While linear matrix inequalities are commonly used for optimal control of Ordinary Differential Equations (ODEs), the absence of a universal state-space representation and suitable dual form prevents such methods from being applied to optimal control of PDEs. Specifically, for ODEs, the controller synthesis problem is defined in state-space, and duality is used to resolve the bilinearity of that synthesis problem. Recently, the PIE representation was proposed as a universal state-space representation for linear PDE systems. In this paper, we show that any PDE system represented by a PIE admits a dual PIE with identical stability and I/O properties. This result allows us to reformulate the stabilizing and optimal state-feedback control problems as convex optimization over the cone of positive Partial Integral (PI) operators. Operator inversion formulae then allow us to construct feedback gains for the original PDE system. The results are verified through application to several canonical problems in optimal control of PDEs and indicate the resulting bounds on $\hinf$ norm are not conservative.
\end{abstract}


\section{Introduction}\label{sec:introduction}
Partial Differential Equations (PDEs) are used to model spatially-distributed phenomena such as vibrations in beams \cite{timoshenko}, turbulent fluid flows \cite{alfonsi2009reynolds}, and reaction kinetics \cite{chakraborty2005spatially,christofides2002nonlinear}. In many systems governed by PDEs, optimal control can significantly improve safety and reduce operational costs. For example: controllers designed for Euler or Timoshenko beam models can suppress seismic and wind disturbances in buildings and bridges~\cite{kannan1995active,ikeda2004active,fisco2011smart} (thereby reducing structural damage); controllers for fluid-flow models can reduce drag on aircraft wings \cite{quadrio2011drag} (thereby reducing fuel costs); and controllers for reaction-diffusion equations can improve homogeneity (or desired stratification) of concentration and temperature in chemical reactors~\cite{mao2017micro,chakraborty2005spatially} (thereby optimizing reaction rates). See~\cite{pesch2012optimal} for a survey on PDE models in optimal control.

Despite or perhaps because of the variety of PDE models, and unlike the state-space transfer function and Linear Matrix Inequality (LMI) methods developed for Ordinary Differential Equations (ODEs), approaches to control of PDEs tend to be ad hoc---applicable only to a specific PDE and set of boundary conditions. Furthermore, because of the limited scope of such existing methods for control of PDEs, and because of the extensive knowledge and expertise needed to extend such methods to application-based, multivariate, data-based PDEs, this traditional ad-hoc approach has not resulted in significant improvement in desired outcomes such as drag reduction, plasma stabilization, or reaction efficiency in the same way that state-space methods for ODEs have dramatically improved safety, manufacturing, autonomy, and efficiency. So what about state-space methods for ODEs enable their efficient utilization in complex, multivariate, data-based models? The answer lies in separation of model from method. Specifically, state-space allows the user to specify a system model in a suitably general form with the understanding that for any model of that form, there already exist efficient numerical algorithms, implemented in reliable software, which can provide simulation, performance analysis, and optimal controllers for that system. The extension of this approach to control of PDEs, then, requires both a sufficiently general notion of state-space for PDEs, as well as a class of numerical methods suitable to that representation. Fortunately, such a state-space representation now exists in the form of the recently proposed Partial Integral Equation (PIE) framework. The goal of this paper, then, is to design methods for controller synthesis which are model agnostic in that they apply to any suitably well-posed PDE system expressed in the PIE framework and they admit efficient numerical algorithms for software implementation.



While the problem of optimal feedback control of PDEs is underdeveloped \cite{troltzsch2024optimal}, efficient algorithms exist for robust and optimal feedback control of linear state-space Ordinary Differential Equations (ODEs), with such controllers typically obtained by solving either Riccati Equations \cite{locatelli2002optimal,gerdts2011optimal} or Linear Matrix Inequalities (LMIs) \cite{boyd1994linear}. Because efficient algorithms exist for controller synthesis of linear state-space ODE systems, the most common approach to control of PDEs is to approximate the PDE model with a lumped state-space ODE model using methods such as frequency-domain analysis and projection~\cite{bamieh2002distributed,apkarian2020boundary,collis2002analysis} or finite-difference~\cite{christofides1996feedback,ito1998optimal,ito1998reduced}. However, stability and performance gains of the closed-loop state-space lumped ODE do not necessarily translate to stability or performance of the optimal closed-loop PDE~\cite{hinze2011discretization,morris2010control} -- PDEs that have a finite number of unstable modes~\cite{prieur2018feedback} are an exception. The downsides of these \textit{early-lumping} methods are: closed-loop stability is not guaranteed; large discretized state-spaces increase computational cost; and implementation requires mapping measurements of the physical system to states of the ODE approximation.

To avoid reducing the PDE model to a linear state-space ODE, one can formulate the optimal control problem in an abstract operator-theoretic state-space framework~\cite{Curtain:1995:IIL:207416} to obtain an operator equivalent of the Riccati Equations for controller synthesis~\cite{lasiecka2000control,morris2001h}. Unfortunately, however, the operators in these Riccati equations are unbounded and cannot be easily parameterized. As a result, one needs to project the operator equations onto a finite-dimensional subspace \cite{aksikas2017optimal} -- implying that the solution is only valid on the projected subspace. This approach is often referred to as \textit{late-lumping} and has associated convergence proofs that show the error of approximation decreases with the increase in the number of bases for the projected subspace. The downsides of late-lumping are: the projection requires extensive ad hoc analysis for any given PDE; the operator solutions are never obtained explicitly (only their projection onto a finite-dimensional subspace); and the closed-loop is not guaranteed to be stable for any given order of projection.

Related to the Riccati-based \textit{late-lumping} approach is Backstepping, which uses a feedback controller to transform the closed-loop dynamics into a form that is equivalent (through an invertible state transformation) to that of a nominal stable system~\cite{krstic2008boundary,meurer2012control,vazquez2024backstepping}. The resulting state transformation is an integral operator whose kernel is defined by a set of PDEs that must be solved numerically and, in certain cases, convergence proofs are available. The downsides of Backstepping are: the kernel map must be re-derived for every PDE (recent work has focused on power series~\cite{vazquez_2023} or neural networks~\cite{bhan_2023} to find the kernel); a parameterization of the kernels is required in order to numerically solve the resulting kernel PDEs; and the controllers are stabilizing, not optimal.

To avoid lumping and numerical solution of kernels, recent work has focused on explicit parameterizations of positive Lyapunov functions, often using positive matrices and LMIs to enforce the positivity of these Lyapunov functions \cite{fridman2009lmi,gahlawat_2016ACC,magron2020optimal}. Since the resulting conditions for stability or performance of the controller are formulated in terms of LMIs, one can use efficient interior-point solvers to solve these LMIs to obtain provable properties of the PDE. The downsides of this approach are: the assumption of specific structure on the Lyapunov function and controller adds conservatism to the problem; the use of ad hoc steps such as Poincar\'e and Wirtinger inequalities to upper bound the derivative of the Lyapunov function; and the failure to resolve the bilinearity between the Lyapunov variable and the feedback gain variable often renders the problem non-convex or severely limits the structure of the Lyapunov function and/or controller.

We conclude, therefore, that existing methods for optimal feedback control of PDEs: lack provable properties, are ad hoc, or are conservative. Our goal, then, is to overcome some of these disadvantages by using a newly developed state-space representation of PDEs to obtain dual representations of the PDE. This dual representation is then used to propose a convex formulation of the $\hinf$-optimal state-feedback controller synthesis problem. This approach has advantages over prior work in that it applies to any suitably well-posed PDE, requires no ad hoc steps or manipulation, and has few obvious sources of conservatism.

Having stated our goal, let us now consider the Partial Integral Equation (PIE) representation of PDE optimal control problems. Specifically, a PIE has the following state-space form
\begin{equation*}
\bmat{\mcl T \dot{\mbf x}(t)\\z(t)\\y(t)}=\bmat{\mcl A&\mcl B_1&\mcl B_2\\\mcl C_1&D_{11}&D_{12}\\\mcl C_2&D_{21}&D_{22}}\bmat{\mbf x(t)\\w(t)\\u(t)}+\mcl B_{w}\dot w(t)+\mcl B_{u}\dot u(t),
\end{equation*}
where $u$ is the control input, $w$ is the exogenous disturbance, $y$ is the measured output, $z$ is the regulated output, and $\mbf x$ is the system state. The operators $\mcl{T}, \mcl{A}, \mcl{B}_i, \mcl{C}_i$ are all bounded integral operators and $D_{ij}$ are matrices. For any suitably well-posed PDE optimal control problem, these operators may be constructed from analytic formulae~\cite{peet_2021AUT} or software implementations such as PIETOOLS~\cite{toolbox:pietools} (See also Section~\ref{sec:PIEs}). The major difference between the PIE and PDE representations is that the state of the PIE is the highest-order spatial derivative of the PDE state (e.g., $\mbf x=\partial_s^2\mbf v$), which is related to the PDE state through an integral operator with polynomial kernel ($\mbf v=\mcl T\mbf x$). Since the state of the PIE is the highest-order spatial derivative of the PDE state, the PIE does not require boundary conditions. Instead, the boundary conditions are implicit in the map $\mcl T:\mbf x\mapsto \mbf v$, and their effect on the dynamics is made explicit in the operators $\mcl A$ and $\mcl B_i$.

The PIE representation retains many of the advantages of the state-space framework used for ODEs. To illustrate, consider stability of a state-space ODE (i.e. $\dot x=Ax$). For $x \in \R^n$, a necessary and sufficient condition for $V(x)$ to be a positive quadratic Lyapunov candidate is that it has the form $V(x)=x^TPx$ for some positive definite matrix $P>0$. More importantly,
square matrices form a linear algebra, and hence the derivative of such Lyapunov candidates are also square matrices, i.e., $A^TP+PA$ is itself a matrix and hence a necessary and sufficient condition for $\dot V(x)=x^T(A^TP+PA)x\le 0$ is that
$A^TP+PA\le 0$. Extending these concepts to PIEs, we find that some vector spaces of integral operators also form a linear algebra. Specifically, we define the class of bounded linear Partial Integral (PI) operators (denoted $\pie\in \PI_4$ -- See Section~\ref{sec:PIEs}) to be those of the form
\[
\left(\fourpi{P}{Q_1}{Q_2}{R_i}\bmat{x\\\mbf \Phi}\right)(s):= \bmat{
Px + \int_{-1}^{0} Q_1(s)\mbf \Phi(s)ds\\
Q_2(s)x +\left(\mcl R\mbf \Phi\right)(s)
}
\] where
{\small
\[
\left(\mcl{R}\mbf \Phi\right)(s)\hspace{-0.5mm}=\hspace{-0.5mm} R_0(s) \mbf \Phi(s) +\hspace{-0.5mm}\int\limits_{-1}^s R_1(s,\theta)\mbf \Phi(\theta)d \theta+\hspace{-0.5mm}\int\limits_s^0R_2(s,\theta)\mbf \Phi(\theta)d \theta.
\]} 
The square elements of this subspace $\PI_4$ form a linear composition algebra. As a result, many of the LMI methods used in the analysis and control of state-space ODEs can be generalized to optimization of positive PI operators. Specifically, if one considers a quadratic Lyapunov function for the PDE $V(\mbf v)=\ip{\mbf v}{\mcl P \mbf v}$ defined on a PDE state $\mbf v=\mcl T\mbf x$, where $\mcl P\succ 0$ is a PI operator, then if $u=w=0$, $\dot V=\ip{\mbf x}{(\mcl T^*\mcl P\mcl A+\mcl A^*\mcl P\mcl T)\mbf x}$ and hence negativity of the PI operator $\mcl T^*\mcl P\mcl A+\mcl A^*\mcl P\mcl T$ is necessary and sufficient for the stability of the PDE\footnote{Note that for integral operators on a Hilbert space, $^*$ represents the adjoint operator  -- i.e., $\ip{x}{\mcl A y} = \ip{\mcl A^*x}{y}$ for all $x,y$. For $\PI$ operators defined here, this Hilbert space is the Cartesian product $\R^m \times L_2^n$.}. Furthermore, the positivity of $\mcl P$ and negativity of $\mcl T^*\mcl P\mcl A+\mcl A^*\mcl P\mcl T$ can be verified using convex optimization solvers embedded in software packages such as PIETOOLS~\cite{toolbox:pietools}. This approach to a generalization of LMI methods to PIEs has previously been used to solve the problems of stability analysis, $L_2$-gain, and optimal estimator design for linear PDEs~\cite{peet2019discussion,shivakumar_2019CDC,das_2019CDC}.

Unlike analysis, however, controller synthesis is fundamentally a non-convex optimization problem, where the non-convexity arises because we are simultaneously searching for both a Lyapunov certificate of system properties and a controller that is being chosen to optimize those system properties. To understand what a dual system is and how it allows for convexification of the optimal control problem, let us first recall results developed for linear state-space ODEs. Specifically, for any state-space ODE system $G=\{A,B,C,D\}$ (Primal ODE\footnote{We use $\{A,B,C,D\}$ to refer to the state-space system with transfer function $C(sI-A)^{-1}B+D$.}), we may define a dual ODE state-space system $G_d=\{A^T,C^T,B^T,D^T\}$ (Dual ODE).
Duality theory establishes that the systems $G$ and $G_d$ have related properties -- e.g. controllability of $G$ implies observability of $G_d$ and vice versa. Furthermore, and of particular importance to optimal control, duality theory shows that $G$ and $G_d$ have identical stability and input-output properties -- e.g. $\norm{G}_{\mcl L(L_2)}=\norm{G_d}_{\mcl L(L_2)}$.

Equipped with this duality result, control of linear state-space ODEs is relatively simple. Specifically, if we want to design a controller $u(t)=Kx(t)$ then, by the KYP lemma applied to the primal closed-loop system, the $\hinf$-norm of the optimal closed-loop system is given by the smallest $\gamma$ such that
\[
\bmat{P(A+BK)+(A+BK)^TP&PB & (C+DK)^T\\B^TP &-\gamma I&D^T\\C+DK&D&-\gamma I}\le 0
\]
for some $K$ and $P>0$ -- a condition which is \textit{bilinear} in $K$ and $P$. However, if we apply the KYP lemma to the \textit{dual} closed-loop system, and define the new variable $Z=KP$, then the optimal closed-loop $\hinf$-norm is given by the smallest $\gamma$  such that
\[
\bmat{(AP+BZ)+(AP+BZ)^T&(CP+DZ)^T & B\\CP+DZ &-\gamma I&D\\B^T&D^T&-\gamma I}\le 0
\]
for some $P>0$ and $Z$ -- a condition which is \textit{linear} in $P$ and $Z$, and where the optimal state-feedback controller gain is recovered as $K=ZP^{-1}$. 

We conclude, therefore, that if we want to solve the problem of optimal control of PIEs (and hence PDEs), we need to be able to construct the dual representation of a PIE and show that this dual representation retains the stability and performance properties of the primal.
\footnote{While dual representations of PDE systems have been studied in the context of the semigroup framework~\cite{delfour1972controllability,pommaret1995duality}, these methods require both extensive ad hoc mathematical analysis and, when applied to controller synthesis, late lumping of the resulting operator equations~\cite{Curtain:1995:IIL:207416}.}
Specifically, we will show that, for a given \textit{primal} PIE (defined by $\mcl G= \{\mcl T,\mcl A,\mcl B,\mcl C,\mcl D\}$\footnote{Analogous to the ODE notation for systems, we use $\{\mcl T,\mcl A,\mcl B,\mcl C,\mcl D\}$ to refer to a system of partial integral equations of the form $\partial_t(\mcl T\mbf x)(t) = \mcl A\mbf x(t)+\mcl Bu(t)$ with output $z(t) = \mcl C\mbf x(t)+\mcl D u(t)$.}, we may define the \textit{dual PIE} as $\mcl G_d=\{\mcl T^*,\mcl A^*,\mcl C^*,\mcl B^*,\mcl D^*\}$ (a simple construction of the adjoint operator is given in Eq.~\eqref{eqn:adjoint}).
Then, in Theorem~\ref{thm:dual_stable}, we show that stability of $\mcl G$ implies stability of $\mcl G_d$ and vice-versa. Furthermore, in Theorem~\ref{thm:dual_hinfty}, we show equivalence of induced norms: $\norm{\mcl G}_{\mcl L(L_2)}=\norm{\mcl G_d}_{\mcl L(L_2)}$\footnote{The Banach space of bounded linear operators on $L_2[0,\infty]$ is denoted $\mcl L(L_2)$ and for ODEs is equivalent to the $H_\infty$ system norm.}.

Equipped with these duality results (Thms.~\ref{thm:dual_stable} and~\ref{thm:dual_hinfty}), generalization of the LMI for $\hinf$-optimal state-feedback control of PIEs becomes relatively simple -- the resulting optimization problem is formulated in Theorem~\ref{cor:optimal_control_lpi} and can be implemented using the PIETOOLS Matlab toolbox~\cite{toolbox:pietools}. We also note that, for actuation at the boundary, the controller synthesis conditions given here require filtering -- See Subsec.~\ref{subsec:note_boundary}.

\section{Notation}\label{sec:notation}
$\R$, $\R^+$ and $\R^n$ denote real numbers, positive real numbers, and n-tuples of real numbers. $L_2^n[a,b]$ denotes the set of $\R^n$-valued, Lebesgue square-integrable equivalence class of functions on spatial domain $[a,b]\subset \R$. 
$\R L_2^{m,n}[a,b]$ denotes the Cartesian product space $\R^{m}\times L_2^{n}[a,b]$  
with inner-product
\begin{align*}
\ip{\bmat{x_1\\\mbf x_2}}{\bmat{y_1\\\mbf y_2}}_{\R L_2} := x_1^Ty_1 + \ip{\mbf x_2}{\mbf y_2}_{L_2}.
\end{align*}
When clear from context, we occasionally omit the subscript on the inner product. (i.e. $\ip{\cdot}{\cdot}:=\ip{\cdot}{\cdot}_{\R L_2}$). Since the only domain considered here is $[a,b]$, we typically omit the domain and simply write $L_2^n$ or $\R L_2^{m,n}$ -- further omitting the dimensions when clear from context. For normed space $X$ and Banach space $Y$, $\mcl L(X,Y)$ denotes the Banach space of bounded linear operators from $X$ to $Y$ with induced norm $\norm{\mcl P}_{\mcl L(X,Y)}=\sup_{\norm{\mbf x}_X=1}\norm{\mcl P\mbf x}_Y$ and where $\mcl L(X):=\mcl L(X,X)$. Our notational convention is to write functions in \textbf{bold} so that $\mbf x$ indicates $\mbf x \in  \R L_2$ and operators in calligraphic capital so $\mcl P$ indicates $\mcl P \in \mcl L(\R L_2)$. For Hilbert space, $X$ and $\mcl A\in \mcl L(X)$, $\mcl A^*$ denotes the adjoint operator satisfying $\ip{\mbf x}{\mcl A\mbf y}_{X} = \ip{\mcl A^*\mbf x}{\mbf y}_{X}$ for all $\mbf x, \mbf y \in X$.

$L_2^n[0,\infty)$ denotes $\R^n$-valued square-integrable signals, where $[0,\infty)$ indicates a temporal domain -- so for $ x \in L_2^n[0,\infty)$, we have $ x(t)\in \R^n$. Similarly, $\R L_2^n[0,\infty)$ denotes $\R L_2$-valued square-integrable signals. 
\blue{For $\mbf x \in \R L_2^n[0,\infty)$, we say $\mbf x(t) \rightharpoonup 0$ if $\mbf x(t)$ converges weakly -- i.e. $\lim_{t\rightarrow \infty} \ip{\mbf x(t)}{\mbf y}= 0$ for all $\mbf y \in \R L_2$.}
For suitably differentiable $\mbf x \in \R L_2^n[0,\infty)$,  $\dot{\mbf x}$ denotes the partial derivative $\frac{\partial \mbf x}{\partial t}$.


\section{A State-Space Framework for Optimal Control of PDEs}\label{sec:PIEs}
Creation of a general framework for the control of PDEs is complicated by the lack of a universal parameterization or representation of such problems. In this section, we examine the class of Partial Integral Equations (PIEs) and use this framework to propose a unified formulation of the stabilization and $\hinf$-optimal state-feedback controller synthesis problems.

PIEs are parameterized by PI operators, which are elements of the algebra, $\PI_4$, defined as follows.
\begin{defn}\label{def:4PI}
We say $\mcl P \in \PI_4 \subset \mcl L(\R L_2^{m_1,n_1},\R L_2^{m_2,n_2})$ if there exists a matrix $P$ and matrix polynomials $Q_1,Q_2,R_0,R_1$, and $R_2$ (of compatible dimension) such that
\begin{align*}
&\mcl P=\fourpi{P}{Q_1}{Q_2}{R_i}\bmat{x\\\mbf{x}}(s) := \bmat{Px + \int_{a}^{b}Q_1(s)\mbf{x}(s)ds\\Q_2(s)x+ \mcl R\mbf{x} (s)},\\
&\left(\mcl R\mbf x\right)(s)\hspace{-1.5mm}= \hspace{-1.5mm}R_0(s) \mbf x(s) +\hspace{-1.5mm}\int\limits_{a}^s  \hspace{-1.5mm}R_1(s,\theta)\mbf x(\theta)d \theta+\hspace{-1.5mm}\int\limits_s^b \hspace{-1.5mm}R_2(s,\theta)\mbf x(\theta)d \theta.
\end{align*}
\end{defn}
The vector space of PI operators, given compatible dimensions, are closed under composition, addition, adjoint and concatenation -- implying that $\PI_4$ is a composition algebra. 
\begin{lem}The vector space $\PI_4\subset \mcl L(\R L_2^{m,n})$ is a $^*$-algebra under composition.
\end{lem}
Significantly, the adjoint operator is readily obtained from the operator parameters as 
\begin{align}
&\fourpi{P}{Q_1(s)}{Q_2(s)}{R_0(s),R_1(s,\theta),R_2(s,\theta)}^*\notag \\
&\qquad =\fourpi{P^T}{Q_2(s)^T}{Q_1(s)^T}{R_0(s)^T,R_2(\theta,s)^T,R_1(\theta,s)^T}.\label{eqn:adjoint}
\end{align}
Similar formulae for composition, addition and concatenation can be found in~\cite{shivakumar_representation_TAC}.


The notation $\fourpi{P}{Q_1}{Q_2}{R_i}$ is widely used throughout this paper to indicate the PI operator associated with the matrix $P$ and polynomial parameters $Q_i$, $R_j$. The dimensions ($m_1,n_1,m_2,n_2$) of the domain ($\R L_2^{m_1,n_1}$) and range ($\R L_2^{m_2,n_2}$) of these operators are inherited from the dimensions of the matrices $P\in \R^{n_2 \times n_1}$ and polynomials $R_0(s)\in \R^{m_2 \times m_1}$. When clear from context, we will omit the dimensions of the domain and range and simply use $\R L_2$. In the case where a dimension is zero, we use $\emptyset$ in place of the associated parameter with dimension zero, so that for example, if $m_1=0$, we have an operator of the form
\[
\fourpi{\emptyset}{\emptyset}{Q_2}{R_i}.
\]
We may now define the class of Partial Integral Equations (PIEs) as follows.
\begin{defn} Given PI operators $\mcl T$, $\mcl A$, $\mcl B_i$, $\mcl C$, matrices $D_i$, and signals $w,u \in L_2[0,\infty)$ of compatible dimensions, we say $\mbf x \in \R L_2[0,\infty)$ and $z \in L_2[0,\infty)$ satisfy the PIE defined by $\{\mcl T,\mcl A,\mcl B_i,\mcl C,D_i\}$ with initial condition $\mbf x_0$ if 
\begin{equation}\label{eq:PIE_full}
\bmat{\partial_t(\mcl{T}\mbf{x})(t)\\z(t)}  = \bmat{\mcl{A}&\mcl{B}_1&\mcl{B}_2\\\mcl C& D_1 &D_2}\bmat{\mbf x(t)\\w(t)\\u(t)}
\end{equation}
for all $t \ge 0$ and $\mcl T \mbf x(0)=\mcl T\mbf x_0$.
\end{defn}
%

The stabilization and $\hinf$-optimal state feedback control problems may now be compactly formulated as follows.\vspace{1mm}
\begin{flalign}
&\textbf{Stabilization:}\;
\text{Find} ~\mcl K \in \PI_4 \quad \text{s.t.},\label{eq:stabilization_PIE}\\
&u(t)=\mcl K \mbf x(t) \; \text{implies} \;\lim_{t\rightarrow \infty}\; \mbf x(t)\rightharpoonup 0\notag&
\end{flalign}
for any $\mbf x(0) \in \R L_2$ and $\mbf x$ which satisfy Eq.~\eqref{eq:PIE_full}.\vspace{1mm}
\blue{\begin{flalign}
&\textbf{$H_\infty$-Optimal State Feedback:}
\;\inf\limits_{\mcl K \in \PI_4,~ \gamma \in \R_+} \quad \gamma \quad \text{s.t.},\label{eq:hinf_opt_PIE}\\
& \;\norm{z}_{L_2}\le \gamma \norm{w}_{L_2}\;  \text{for any $w\in L_2$ with $\norm{w}_{L_2}< M$}&\notag
\end{flalign}
when $z$ satisfies Eq.~\eqref{eq:PIE_full} for some $\mbf x$ with $u(t)=\mcl K \mbf x(t)$ and $\mcl T \mbf x(0)=0$. The choice of bound, $M \in \R^+$ is arbitrary.

Note that the optimal control framework here is chosen to mirror the classical problem of $H_\infty$-optimal control of state-space ODEs. Recent work, however, has examined the effect of including non-zero initial conditions in the problem formulation~\cite{chen2025lyapunov,coutinho2024l2}. In this case, performance would be measured using a control objective augmented with a term of the form of either $\gamma' \norm{\mcl T \mbf x(0)}^2$ or $\gamma' \norm{\mbf x(0)}^2$. This choice of augmented norm will then significantly affect the measured performance~\cite{krsticIO} and hence solution of the optimal control problem. To avoid such complications, then, we do not consider state to output performance at present. }
%

\subsection{Representation of the PDE Control Problem Using PIEs}
A broad class of linear delayed and partial differential equations on a rectangular domain have been shown to admit a PIE representation of the form of Eq.~\eqref{eq:PIE_full}. To avoid unnecessary digression, we will refer to~\cite{shivakumar_representation_TAC} for the full class of linear PDEs that admit such a representation. However, to better illustrate the scope of the results, we note that for any well-posed, vector-valued PDE of the following form, the formulae in~\cite{shivakumar_representation_TAC} (or Matlab interface in~\cite{toolbox:pietools}) may be used to construct a PIE representation of the form $\{\mcl T,\mcl A,\mcl B_i,\mcl C,D_{ij}\}$ where $\mcl T,\mcl A,\mcl B_i,\mcl C \in \PI_4$.
\begin{align}\label{eq:PDE_n}
    \dot{\mbf x}(t,s) &= \sum_{i=0}^n A_i(s)\partial_s^i\mbf x(t,s)+B_1(s)w(t)+B_2(s) u(t),\notag\\
    z(t) &= \int_a^b \sum_{i=0}^{n}C_i(s)\partial_s^i \mbf x(t,s)ds + D_1w(t) + D_2 u(t),\notag\\
    \sum_{i=0}^{n-1} &N_{i} \partial_s^i \mbf x(t,0)+\sum_{i=0}^{n-1} M_{i} \partial_s^i \mbf x(t,1)=0,
\end{align}
Here $\mbf x(t)\in L_2^m[0,1]$ and the parameters $A_i, B_i,$ and $C_i$ are polynomials and $D_i, N_i,$ and $M_i$ are matrices\footnote{Note that because this is a restricted class (excluding coupled ODEs), the parameters of PDE given here are not necessarily the same as used for the more general formulation in~\cite{shivakumar_representation_TAC} or~\cite{toolbox:pietools}. To avoid unnecessary details of this construction and interpretation of parameters, a software user interface has been created in PIETOOLS~\cite{toolbox:pietools} which automates the conversion of PDE or time-delay system to PIE.}. Note that for simplicity, we do not include coupled ODE-PDE systems in this class (as required for time-delay systems). However, the class of PIEs considered allows for inclusion of ODEs coupled to the PDE either in the domain or at the boundary.



While for non-trivial systems the conversion of a PDE to the associated PIE is most readily performed using the PIETOOLS software interface~\cite{toolbox:pietools}, the following example may aid in understanding the relationship between a PDE and its PIE representation.
\begin{ex}\label{ex:EB-representation}
Consider the vibration suppression problem for a cantilevered Euler-Bernoulli beam: $\ddot{\mbf u}=-0.1 {\partial_s^4\mbf u}+w(t)+u(t)$ with $0={\mbf u}(0)={\mbf u}_s(0)={\partial_s^2\mbf u}(1)={\partial_s^3\mbf u}(1)$ where ${\mbf u}$ is displacement, $w$ is external disturbance, and $u$ is controller input. The regulated output is defined as $z(t) = \bmat{u(t) &{\int_0^1\mbf u}(t,s)ds}^T$.

As discussed in~\cite{peet_2021AUTb}, to put this PDE in first-order form, we may define  $\mbf{v}_1:= \dot{\mbf u}$ and $\mbf v_2:= {\partial_s^2\mbf u}$, which yields
\begin{align*}
&\bmat{\dot{\mbf{v}}_1(t,s)\\\dot{\mbf{v}}_2(t,s)} = \bmat{0&-0.1\\1&0}\bmat{\partial_s^2\mbf{v}_{1}(t,s)\\\partial_s^2 \mbf{v}_{2}(t,s)}+\bmat{1\\0}w(t)+\bmat{1\\0}u(t),\\
&\mbf v_1(t,0) = \partial_s\mbf v_1(t,0) = \mbf v_{2}(t,1) =\partial_s\mbf v_{2}(t,1) = 0.
\end{align*}

To convert to a PIE, we use Cauchy's rule for repeated integration to obtain
\[
\mbf v(t,s)=\mbf v(t,0)+s\partial_s\mbf v(t,0)+\int_0^s (s-\theta)\partial_s^2\mbf v(t,\theta)d\theta.
\]
Denoting $\mbf x:=\partial_s^2 \mbf v$ and substituting boundary conditions, we obtain the map from PIE state, $\mbf x$ to PDE state, $\mbf v=\mcl T \mbf x$, as
\begin{equation*}
\bmat{\mbf v_1(t,s)\\ \mbf v_2(t,s)}\hspace{-1mm} =\hspace{-1mm} \int\limits_0^s \bmat{(s-\theta)&\hspace{-2mm}0\\0&\hspace{-2mm}0}\mbf x(t,\theta)d\theta + \int\limits_s^1 \bmat{0&\hspace{-2mm}0\\0&\hspace{-2mm}(\theta-s)}\mbf x(t,\theta)d\theta.
\end{equation*}
%
Substituting this expression into the dynamics and output equation, we obtain the PIE representation
\begin{align*}
&\partial_t\left(\mcl T \mbf x(t)\right)(s) = \overbrace{\bmat{0&-0.1\\1&0}}^{\mcl A}\mbf{x}(t,s)+\overbrace{\bmat{1\\0}}^{\mcl B_1}w(t)+\overbrace{\bmat{1\\0}}^{\mcl B_2}u(t).
\end{align*}
For the regulated output, we use a similar expansion on $\mbf u$, 
\[
\mbf u(s)\hspace{-1mm} = \hspace{-1mm}\mbf u(0) + s\partial_s\mbf u(0)+\int_0^s (s-\theta)\partial_s^2\mbf u(\theta)d\theta\hspace{-1mm}=\hspace{-1mm}\int_0^s (s-\theta) \mbf v_2(\theta)d\theta,
\]
to obtain through a change of order of integration
\begin{align*}
    z_2(t) &= \int_0^1 \mbf u(s) ds=\int_0^1 \half (1-s)^2\mbf v_2(t,s) ds.
\end{align*}
Now from $\mbf v_2(t,s)=\int_s^1 (\theta-s)\mbf x(\theta)d \theta$ through another change of order of integration
\[
z(t)=\underbrace{\int_0^1 \bmat{0&0\\0&\hspace{-.5mm}\frac{s^2}{4}-\frac{s^3}{6}+\frac{s^4}{24}}}_{\mcl C}\mbf{x}(t,s) ds+\underbrace{\bmat{1\\0}}_{D_2}u(t).
\]
To illustrate the $\mbf \Pi_4$ notation, we may also write the system parameters as
\begin{align*}
&\mcl T=\fourpi{\emptyset}{\emptyset}{\emptyset}{0,R_1,R_2}, \mcl A = \fourpi{\emptyset}{\emptyset}{\emptyset}{R_0,0,0}, \\
&\mcl B_i = \fourpi{\emptyset}{\emptyset}{Q_2}{\emptyset}, \mcl C = \fourpi{\emptyset}{Q_1}{\emptyset}{\emptyset}, D_2 = \bmat{1\\0},
\end{align*}
where
\begin{align*}
&R_1(s,\theta) \hspace{-0.75mm}=\hspace{-0.75mm} \bmat{s-\theta&0\\0&0},~R_2(s,\theta) \hspace{-0.75mm}=\hspace{-0.75mm} \bmat{0&0\\0&\theta-s},~Q_2 =\bmat{1\\0},\\
&R_0(s) \hspace{-0.75mm}=\hspace{-0.75mm} \bmat{0&\hspace{-1.5mm}-0.1\\1&\hspace{-1.5mm}0},~Q_1(s) \hspace{-0.75mm}=\hspace{-0.75mm} \bmat{0&\hspace{-1.5mm}0\\0&\hspace{-1.5mm}\frac{s^2}{4}-\frac{s^3}{6}+\frac{s^4}{24}}.
\end{align*}

\end{ex}

\subsection{A Note on Inputs at the Boundary}\label{subsec:note_boundary}
When the control input enters the dynamics of a PDE through the boundary conditions (e.g. $\mbf v(t,0)=u(t)$), novel questions arise that are not readily apparent in the PDE representation but are made explicit when using the PIE framework. These questions arise because PDEs with distributed states are partly rigid -- i.e. they are constrained by the continuity properties of Sobolev space necessary for boundary values to be well defined. The simplest illustration of this is the heat equation with boundary conditions $\mbf v(t,0)=u_1(t)$ and $\mbf v_s(t,0)=u_2(t)$. In this case, we have the relationship
\[
\mbf v(t,s)=u_1(t)+s u_2(t)+\int_0^s(s-\eta)\partial_s^2\mbf v(t,\eta)d\eta
\]
which implies that the effect of the input is felt \textit{immediately} throughout the distributed state and is NOT filtered through the dynamics (as is the case in ODEs or in-domain control). If we integrate this type of \textit{semi-algebraic} relationship into the dynamics, we obtain a unitary PIE representation of the heat equation, defined in terms of PIE state $\mbf x =  \partial_s^2\mbf v$, as
\[
\partial_t\left(\int_0^s(s-\eta)\mbf x(t,\eta)d\eta\right)=\mbf x(t,s)-\dot u_1(t)-s\dot u_2(t).
\]
In this representation, the partially algebraic nature of the boundary conditions is made explicit in that the dependence is not on $u_1, u_2$, but on their time-derivatives. 
This type of dependence is allowed in the parameterization of PIE defined in~\cite{shivakumar_representation_TAC} but is not included in the controller synthesis approach defined here. One reason is that there is a valid argument to be made that such types of control are non-physical in that they do not account for the inertia of the distributed state, and hence, such inputs would be better modeled by filtering through an ODE which represents the dynamics of the actuator. The other reason is that if we are searching for an $\hinf$-optimal controller, then we are trying to minimize the gain from $\norm{w}_{L_2}$ to $\norm{z}_{L_2}$ and if we were to include the derivative $\dot w$, this implies that the natural norm for $w$ is the Sobolev norm --  an approach taken in~\cite[Thm. 3.3]{Curtain:1995:IIL:207416}.

Therefore, to account for the case of inputs at the boundary, we will assume that the actual disturbance or input signal is not $w$ or $u$, but rather $\dot w,\dot u$, which we can relabel as disturbances $\hat w,\hat u$. This approach allows us to take any PIE optimal control problem involving time derivatives of the inputs and reformulate it as a PIE free of such derivatives. Specifically, if we are given a PIE representation of the form
\[
\bmat{\partial_t(\mcl{T}\mbf{x}(t))\\z(t)}  = \bmat{\mcl{A}&\mcl{B}_1&\mcl{B}_2\\\mcl C& D_1 &D_2}\bmat{\mbf x(t)\\w(t)\\u(t)} + \mcl B_{1d}\dot w(t)+ \mcl B_{2d}\dot u(t),
\]
then we will augment the state $\hat{\mbf x}(t):= \bmat{w(t)\\u(t)\\\mbf x(t)}$ and redefine the PIE system as{\small 
\begin{align*}
&\bmat{\partial_t\bbbbl(\overbrace{\bmat{I&0&0\\0&I&0\\0&0&\mcl{T}}}^{\hat{\mcl T}}\hat{\mbf x}(t)\bbbbr)\\\vspace{-2mm}\\z(t)}  \notag\\
&\qquad \qquad = \bmat{\overbrace{\bmat{0&0&0\\0&0&0\\\mcl{B}_1&\mcl{B}_2&\mcl{A}}}^{\hat{\mcl A}}&\overbrace{\bmat{I\\0\\\mcl B_{1d}}}^{\hat{\mcl B}_1}&\overbrace{\bmat{0\\I\\\mcl B_{2d}}}^{\hat{\mcl B}_2}\\\vspace{-2mm}\\
\underbrace{\bmat{D_1&D_2&\mcl C}}_{\hat{\mcl C}}&\underbrace{\bmat{~~0~~}}_{\hat D_1}&\underbrace{\bmat{~~0~~}}_{\hat D_2}}\bmat{\hat{\mbf x}(t)\\\hat w(t)\\\hat u(t)}
\end{align*}}
which is now of the form in Eq.~\eqref{eq:PIE_full} using the parameters $\hat{\mcl T}$, $\hat{\mcl A}$, $\hat{\mcl B}_1$, etc. Numerical examples of such boundary control problems are included in Section~\ref{sec:numerical} as Examples \ref{ex:hinf_diffusion} and \ref{ex:hinf_wave}.

\subsection{A Three-Stage Approach to Controller Synthesis}
Given a PIE formulation of the stabilizing and $H_\infty$-optimal state-feedback problems in Eqs.~\eqref{eq:stabilization_PIE} and~\eqref{eq:hinf_opt_PIE}, we will solve these problems in three stages.

First, for a given PIE, we define a dual PIE and show that this PIE has identical internal stability and input-output properties -- Section~\ref{sec:dual}. Second, in Section~\ref{sec:analysis}, we define a class of optimization problems defined by linear PI operator inequality (LPI) constraints for which we have efficient convex optimization algorithms. We use this operator inequality (LPI) framework to solve the problems of internal stability and input-to-output (I/O) $L_2$-gain performance in Subsections~\ref{subsec:LPI_stab} and~\ref{subsec:LPI_kyp}, respectively.  Third, in Section~\ref{sec:noboundary_control}, we apply the results of Section~\ref{sec:analysis} to the dual of the closed-loop PIE. We then use variable substitution to formulate the stabilization and $\hinf$-optimal state-feedback control problems as LPI problems. A formula for inversion of PI operators is presented in Section~\ref{sec:controller} and used to obtain the controller gains. Because solutions of the PIE and the PDE it represents are equivalent, these controller gains can then be applied directly to the PDE model -- a process illustrated in the examples in Section~\ref{sec:numerical}.

\section{Duality in PIEs} \label{sec:dual}
In this section, we show that for any PIE system of the form in Eq.~\eqref{eq:PIE_full}, we may associate a dual PIE system, also of the same form.

\begin{defn}\label{def:dualPIE}(Dual PIE)
Given a PIE system of the form
\begin{equation}\label{eq:pie_gen}
\bmat{\partial_t(\mcl{T}\mbf{x}(t))\\z(t)}  = \bmat{\mcl{A}&\mcl{B}\\\mcl C& D}\bmat{\mbf x(t)\\w(t)},\quad \mbf x(0)\in \R L_2^{m,n},
\end{equation}
defined by PI operators $\mcl{T}$, $\mcl{A}$, $\mcl{B}$, $\mcl{C}$ and matrix $D$, we define the ``dual PIE system" as
\begin{equation}\label{eq:pie_gen_dual}
\bmat{\mcl{T}^*\dot{\bar{\mbf{x}}}(t)\\\bar z(t)}  = \bmat{\mcl{A}^*&\mcl{C}^*\\\mcl B^*& D^T}\bmat{\bar{\mbf x}(t)\\\bar w(t)},\quad \bar{\mbf x}(0)\in \R L_2^{m,n},
\end{equation}
where $^*$ represents the adjoint of an operator with respect to the $\R L_2$ inner product.
\end{defn}
Given polynomial parameters defining the operators $\mcl{T}, \mcl{A}$, $\mcl{B}$, $\mcl{C}$, the polynomials that parameterize the dual PIE operators are readily obtained from Eq.~\eqref{eqn:adjoint}. 

Note that while the primal PIE in Eq.~\eqref{eq:pie_gen} is defined using $\partial_t(\mcl T\mbf x)$, its dual in Eq.~\eqref{eq:pie_gen_dual} is defined in terms of $\mcl T \dot{\bar{\mbf x}}$. This asymmetry may be relaxed if the primal solutions are assumed to be sufficiently differentiable. Having defined the dual of a PIE representation,  the following subsections prove primal-dual equivalence in terms of stability and $L_2$-gain. 


\subsection{Dual Stability Theorem}\label{subsec:dual-stab}
First, let us look at the internal stability of a PIE and its dual. 

\begin{thm}\label{thm:dual_stable}
(Dual Stability) Suppose $\mcl{T},\mcl{A}\in \mcl L(\R L_2^{m,n})$ are PI operators. The following statements are equivalent.
\begin{enumerate}[label=\alph*)]
\item $\lim\limits_{t\to \infty} \mcl{T}\mbf x(t)\rightharpoonup 0$ for any $\mbf x$ that satisfies $\partial_t(\mcl{T}\mbf{x}(t)) = \mcl{A}\mbf x(t)$ with initial condition $\mbf x(0)\in \R L_2^{m,n}$.
\item $\lim\limits_{t\to \infty} \mcl{T}^*\bar{\mbf x}(t)\rightharpoonup 0$ for any $\bar{\mbf x}$ that satisfies $\mcl{T}^*\dot{\bar{\mbf x}}(t) = \mcl{A}^*\bar{\mbf x}(t)$ with initial condition $\bar{\mbf x}(0)\in \R L_2^{m,n}$.
\end{enumerate}
\end{thm}
Note that ``$\rightharpoonup 0$'' in Thm.~\ref{thm:dual_stable} denotes weak convergence in the Hilbert space $\R L_2$.

\begin{proof}
To show sufficiency, suppose $\mbf x$ satisfies $\partial_t(\mcl{T}\mbf{x}(t)) = \mcl{A}\mbf x(t)$ with initial condition $\mbf x(0)\in \R L_2^{m,n}$ and $\lim_{t\rightarrow \infty}\mcl{T}\mbf{x}(t)\rightharpoonup 0$. Let $\bar{\mbf x}$ satisfy $\mcl{T}^*\dot{\bar{\mbf x}}(t) = \mcl{A}^*\bar{\mbf x}(t)$ with initial condition $\bar{\mbf x}(0)\in \R L_2^{m,n}$. Then for any finite $t>0$, using integration-by-parts, we get
{
\begin{align}
&\int_{0}^{t} \ip{\bar{\mbf x}(t-s)}{\partial_s(\mcl{T}\mbf{x}(s))}_{\R L_2}ds\label{eqn:integral1}\\
&\qquad=	\ip{\bar{\mbf x}(0)}{\mcl T\mbf x(t)}-\ip{\bar{\mbf x}(t)}{\mcl{T}\mbf x(0)}  \notag\\
&\qquad\qquad- \int_{0}^t \ip{\partial_s{\bar{\mbf x}}(t-s)}{\mcl{T}\mbf x(s)}ds.\notag
\end{align}}
Then, we use a change of variable ($\theta = t-s$) on the last term in Eq.~\eqref{eqn:integral1} to show
\begin{align*}
\int_{0}^t \ip{\partial_s{\bar{\mbf x}}(t-s)}{\mcl{T}\mbf x(s)}ds\hspace{-.5mm}&=\hspace{-1.5mm}\int_{0}^t \ip{\dot{\bar{\mbf x}}(\theta)}{\mcl{T}\mbf x(t-\theta)}d\theta\\&\hspace{-0.5mm}=\hspace{-1.5mm} \int_{0}^t \ip{\mcl T^*\dot{\bar{\mbf x}}(\theta)}{\mbf x(t-\theta)}d\theta.
\end{align*}
Furthermore, using the same variable change on the left-hand side of Eq.~\eqref{eqn:integral1}, we get
\begin{align*}
&\int_{0}^{t} \ip{\bar{\mbf x}(t-s)}{\partial_s(\mcl{T}\mbf{x}(s))}_{\R L_2}ds\\
&=\int_{0}^{t} \ip{\bar{\mbf x}(t-s)}{\mcl A\mbf x(s)}ds=\int_{0}^{t} \ip{\mcl A^*\bar{\mbf x}(\theta)}{\mbf x(t-\theta)}d\theta.
\end{align*}
Substituting these two expressions into Eq.~\eqref{eqn:integral1}, we have
\begin{align*}
&\int_{0}^{t} \ip{\mcl A^*\bar{\mbf x}(\theta)}{\mbf x(t-\theta)}d\theta \\
&\qquad=\ip{\bar{\mbf x}(0)}{\mcl{T}\mbf x(t)}-\ip{\bar{\mbf x}(t)}{\mcl{T}\mbf x(0)} \\
&\qquad\qquad+ \int_{0}^t \ip{\mcl T^*\dot{\bar{\mbf x}}(\theta)}{\mbf x(t-\theta)}d\theta.
\end{align*}
However, $\mcl A^*\bar{\mbf x}(\theta) = \mcl T^*\dot{\bar{\mbf x}}(\theta)$ for all $\theta\in [0,t]$, and hence
\begin{align}\label{eq:tmp}
\ip{\bar{\mbf x}(0)}{\mcl{T}\mbf x(t)}=\ip{\mcl T^*\bar{\mbf x}(t)}{\mbf x(0)},~~~ \text{for all} ~t>0.
\end{align}
Since $\lim_{t\rightarrow \infty}\mcl{T}\mbf x(t)= 0$, we have
\[
\lim_{t\rightarrow \infty}\ip{\mcl T^*\bar{\mbf x}(t)}{\mbf x(0)}= \lim_{t\rightarrow \infty}\ip{\bar{\mbf x}(0)}{\mcl{T}\mbf x(t)}= 0.\]
Since ${\mbf x}(0)\in \R L_2$ may be chosen arbitrarily and $\R L_2$ is Hilbert, this implies $\lim_{t\rightarrow \infty}\mcl T^*\bar{\mbf x}(t) \rightharpoonup 0$. Thus we have sufficiency. 

Since $\mcl T^{**}=\mcl T$, and existence of $\mcl T^*\dot{\bar{\mbf x}}(t)$ guarantees existence of $\partial_t(\mcl T^*\bar{\mbf x}(t))$, sufficiency implies necessity.
\end{proof}

Next, we show that exponential stability of the primal and dual PIE are equivalent when we define exponential stability in the following sense.
\begin{defn}[Exponential Stability]\label{defn:expstability}
We say that the PIE  defined by $\{\mcl T, \mcl A\} \subset \PI_4$ is \textbf{Exponentially Stable} with decay rate $\alpha >0$ if there exists some $M>0$ such that for any $\mbf x_0 \in \R L_2$, if $\mcl T\mbf x(0)=\mcl T\mbf x_0$ and
$\partial_t(\mcl T\mbf x(t)) = \mcl A\mbf x(t)$, then
\[
\norm{\mcl T\mbf x(t)}_{\R L_2}\le M\norm{\mbf x_0}_{\R L_2}e^{-\alpha t}\quad \text{for all}~t \ge 0.
\]
\end{defn}
Note that this definition implies that the PDE state ($\mcl T\mbf x$) decays exponentially in the $\R L_2$-norm but does not necessarily guarantee exponential stability of the PIE state ($\mbf x$).

\begin{cor}\label{cor:exponential_stable}
Suppose $\mcl{T},\mcl{A}\in \mcl L(\R L_2^{m,n})$ are PI operators. The following statements are equivalent:
\begin{enumerate}[label=\alph*)]
\item $\norm{\mcl T\mbf x(t)}\le M\norm{\mbf x(0)}e^{-\alpha t}$
for any $\mbf x$ that satisfies $\partial_t(\mcl{T}\mbf x)(t) = \mcl{A}\mbf x(t)$ with initial condition $\mbf x(0)\in \R L_2^{m,n}$.
\item $\norm{\mcl T^*\bar{\mbf x}(t)}\le M\norm{\bar{\mbf x}(0)}e^{-\alpha t}$
for any $\bar{\mbf x}$ that satisfies $\mcl{T}^*\dot{\bar{\mbf x}}(t) = \mcl{A}^*\bar{\mbf x}(t)$ with initial condition $\bar{\mbf x}(0)\in \R L_2^{m,n}$.
\end{enumerate}
\end{cor}
\begin{proof}
To show sufficiency,
suppose $\bar{\mbf x}$ satisfies $\mcl{T}^*\dot{\bar{\mbf x}}(t) = \mcl{A}^*\bar{\mbf x}(t)$ for some initial condition $\bar{\mbf x}(0)\in \R L_2^{m,n}$. Then for any $T>0$, let $\mbf x$ satisfy $\partial_t(\mcl{T}\mbf x(t)) = \mcl{A}\mbf x(t)$ with initial condition $\mbf x(0) = \mcl T^* \bar{\mbf x}(T)$. Then, we have from Eq. \eqref{eq:tmp} in the proof of Theorem~\ref{thm:dual_stable},
\begin{align*}
\ip{\bar{\mbf x}(0)}{\mcl T\mbf x(T)}=\ip{\mcl T^*\bar{\mbf x}(T)}{\mbf x(0)} = \norm{\mcl T^*\bar{\mbf x}(T)}^2.
\end{align*}
Then, from the Cauchy-Schwarz inequality,
\begin{align*}
\norm{\mcl T^*\bar{\mbf x}(T)}^2 &= \ip{\bar{\mbf x}(0)}{\mcl T\mbf x(T)} \le \norm{\mcl T\mbf x(T)}\norm{\bar{\mbf x}(0)}\\
&\le M\norm{\mbf x(0)}e^{-\alpha T}\norm{\bar{\mbf x}(0)} \\
&= M\norm{\mcl T^*\bar{\mbf x}(T)}e^{-\alpha T}\norm{\bar{\mbf x}(0)},
\end{align*}
which implies $\norm{\mcl T^*\bar{\mbf x}(T)}\le M\norm{\bar{\mbf x}(0)}e^{-\alpha T}$.
Since $T$ is arbitrary, we have sufficiency. Since $\mcl T^{**}=\mcl T$, sufficiency implies necessity.
\end{proof}

\subsection{Dual $L_2$-Gain Theorem}\label{subsec:dual-hinfty}
Having proved that the internal stability of a PIE and its dual are equivalent, we now show that their I/O properties are identical (analogous to the dual KYP lemma for ODE systems).

\begin{thm}\label{thm:dual_hinfty}
(Dual $L_2$-gain)	Suppose $\mcl{T},\mcl{A}\in \mcl L(\R L_2^{m,n})$, $\mcl{B}\in \mcl L(\R^p,\R L_2^{m,n})$, and $\mcl{C}\in \mcl L(\R L_2^{m,n},\R^r)$ are PI operators and $D\in \R^{r\times p}$ is a matrix. The following statements are equivalent.
\begin{enumerate}[leftmargin=*,labelsep=2pt,label=\alph*)]
\item For any $w\in L_2^p[0,\infty)$, $\mbf x(t)\in \R L_2^{m,n}$ and $z(t)\in \R^r$ that satisfy $\mcl T\mbf x(0)=0$ and
\begin{align}\label{eq:PIE_IO}
\bmat{\partial_t(\mcl{T}\mbf x)(t)\\z(t)}  = \bmat{\mcl{A}&\mcl{B}\\\mcl C& D}\bmat{\mbf x(t)\\w(t)},
\end{align}
we have that $\norm{z}_{L_2}\le \gamma\norm{w}_{L_2}$.
\item For any $\bar{w}\in L_2^r[0,\infty)$, $\bar{\mbf x}(t)\in \R L_2^{m,n}$ and $\bar{z}(t)\in \R^p$ that satisfy $\mcl T^*\bar{\mbf x}(0)=0$ and
\begin{align}\label{eq:PIE_IO_adjoint}
\bmat{\mcl{T}^*\dot{\bar{\mbf{x}}}(t)\\\bar z(t)}  =
\bmat{\mcl{A}^*&\mcl{C}^*\\\mcl B^*& D^T}
\bmat{\bar{\mbf x}(t)\\\bar w(t)},
\end{align}
we have that $\norm{\bar{z}}_{L_2}\le \gamma\norm{\bar w}_{L_2}$.
\end{enumerate}
\end{thm}
\begin{proof}
To show sufficiency, let $\mbf x(t)\in \R L_2^{m,n}$ and $z(t)\in \R^r$ satisfy Eq.~\eqref{eq:PIE_IO} for $\mbf x(0)=0$ and some $w\in L_2^p[0,\infty)$. Then, $\norm{z}_{L_2}\le \gamma\norm{w}_{L_2}$. Let $\bar{\mbf x}(t)\in \R L_2^{m,n}$ and $\bar{z}(t)\in \R^p$ satisfy Eq.~\eqref{eq:PIE_IO_adjoint} for $\bar{\mbf x}(0)=0$ and some $\bar{w}\in L_2^r[0,\infty)$. Then, as in Eq.~\eqref{eqn:integral1} in the proof of Theorem~\ref{thm:dual_stable} and substituting initial conditions, we find
\begin{align}
&\int_{0}^{t}\hspace{-1.5mm} \ip{\bar{\mbf x}(t-s)}{\partial_s(\mcl{T}\mbf x(s))}\hspace{-0.75mm}ds\hspace{-0.75mm}=\hspace{-0.75mm} \int_{0}^t\hspace{-1.5mm} \ip{\mcl T^*\dot{\bar{\mbf x}}(\theta)}{\mbf x(t-\theta)}\hspace{-0.75mm}d\theta.\hspace{3.5mm}(\star)\notag
\end{align}
Furthermore, by using the variable change $\theta=t-s$ on the left-hand side of the above equation,
\begin{align}
&\int_{0}^{t} \ip{\bar{\mbf x}(t-s)}{\partial_s(\mcl{T}\mbf x(s))}ds \hspace{43mm} (\#)\notag\\
&=\hspace{-1.5mm}\int_{0}^{t}\ip{\bar{\mbf x}(t-s)}{\mcl A\mbf x(s)}ds+\int_{0}^{t} \ip{\bar{\mbf x}(t-s)}{\mcl Bw(s)}ds\notag\\
&=\hspace{-1.5mm}\int_{0}^{t} \ip{\mcl A^*\bar{\mbf x}(\theta)}{\mbf x(t-\theta)}d\theta+\int_{0}^{t} (\mcl B^*\bar{\mbf x}(\theta))^Tw(t-\theta)d\theta.\notag
\end{align}
Combining Eqs.~($\star$) and ($\#$), we obtain
\begin{align*}
&\int_{0}^t \ip{\mcl T^*\dot{\bar{\mbf x}}(\theta)}{\mbf x(t-\theta)}d\theta\\
&= \int_{0}^{t} \ip{\mcl A^*\bar{\mbf x}(\theta)}{\mbf x(t-\theta)}d\theta  + \int_{0}^{t} (\mcl B^*\bar{\mbf x}(\theta))^Tw(t-\theta)d\theta.
\end{align*}
However, $\mcl{T}^*\dot{\bar{\mbf x}}(t) -\mcl{A}^*\bar{\mbf x}(t)= \mcl{C}^*\bar{w}(t)$, so
\begin{align*}
&\int_{0}^t \ip{\mcl{C}^*\bar{w}(\theta)}{\mbf x (t-\theta)}d\theta \\
&= \int_{0}^t \ip{\mcl{T}^*\dot{\bar{\mbf x}}(\theta) -\mcl{A}^*\bar{\mbf x}(\theta)}{\mbf x (t-\theta)}d\theta\\
&\qquad \qquad = \int_{0}^{t} (\mcl B^*\bar{\mbf x}(\theta))^Tw(t-\theta)d\theta.
\end{align*}
Since $z = \mcl C\mbf x+D w$, we obtain
\begin{align*}
&\int_{0}^{t} \bar{w}(\theta)^T z(t-\theta)d\theta-\int_{0}^{t} \bar{w}(\theta)^T(Dw(t-\theta))d\theta\\
&=\int_{0}^{t} \bar{w}(\theta)^T(\mcl C \mbf x(t-\theta))d\theta = \int_{0}^t \ip{\mcl{C}^*\bar{w}(\theta)}{\mbf x(t-\theta)}d\theta\\
&= \int_{0}^{t} (\mcl B^*\bar{\mbf x}(\theta))^Tw(t-\theta)d\theta.
\end{align*}
Likewise, we know $\bar z = \mcl B^*\bar{\mbf x}+ D^T\bar{w}$. Hence
\begin{align*}
&\int_{0}^t \bar{z}(\theta)^Tw(t-\theta)d\theta - \int_{0}^{t} D^T\bar{w}(\theta)^Tw(t-\theta)d\theta\\
&= \int_{0}^{t} (\mcl B^*\bar{\mbf x}(\theta))^Tw(t-\theta)d\theta\\
&=\int_{0}^{t} \bar{w}(\theta)^T z(t-\theta)d\theta-\int_{0}^{t} \bar{w}(\theta)^T(Dw(t-\theta))d\theta.
\end{align*}
We conclude that for any $t>0$, if $z$ and $w$ satisfy the primal PIE and $\bar{z}$ and $\bar w$ satisfy the dual PIE, then
\begin{align}\label{eq:dual_io}
\int_{0}^{t} \bar{z}(\theta)^T w(t-\theta)d\theta = \int_{0}^t \bar{w}(\theta)^Tz(t-\theta)d\theta.
\end{align}
Now, for any $\bar w \in L_2$, let $\bar z$ solve the dual PIE for some $\bar{\mbf x}$. For any fixed $T>0$, define $w(t) = \bar z(T-t)$ for $t\le T$ and $w(t)=0$ for $t >T$. Then $w \in L_2$ and for this input, let $z$ solve the primal PIE for some $\mbf x$. Then, if we define the truncation operator $P_T$, we have
\begin{align*}
\norm{P_T\bar z}_{L_2}^2 &= \int_{0}^T \bar z(s)^T\bar z(s)ds= \int_{0}^T \bar z(s)^T w(T-s)ds \\
&= \int_{0}^T \bar w(s)^T z(T-s)ds\le \norm{P_T\bar w}_{L_2}\norm{P_T z}_{L_2}\\
&\le\norm{P_T\bar w}_{L_2}\norm{z}_{L_2}\le\gamma \norm{P_T\bar w}_{L_2}\norm{w}_{L_2}\\
&= \gamma\norm{P_T\bar w}_{L_2}\norm{P_T w}_{L_2} = \gamma\norm{P_T\bar w}_{L_2}\norm{P_T \bar z}_{L_2}.
\end{align*}
Therefore, we have that $\norm{P_T\bar z}_{L_2} \le \gamma\norm{P_T\bar w}_{L_2}$ for all $T\ge 0$. We conclude that $\norm{\bar z}_{L_2} \le \gamma  \norm{\bar w}_{L_2}$. Since $\mcl T^{**}=\mcl T$ and
\[
\bmat{(\mcl{A}^{*})^*&(\mcl{B}^{*})^*\\(\mcl C^{*})^*& (D^{T})^T}=\bmat{\mcl{A}^*&\mcl{C}^*\\\mcl B^*& D^T}^*=\bmat{\mcl{A}&\mcl{B}\\\mcl C& D}^{**}=\bmat{\mcl{A}&\mcl{B}\\\mcl C& D}
\]
we have that sufficiency implies necessity.
\end{proof}

\begin{rem}\label{rem:intertwining}
Note the relationship between primal and dual mappings $w \mapsto z$ and $\bar w \mapsto \bar z$ as given in Eq.~\eqref{eq:dual_io} of the proof of Theorem~\ref{thm:dual_hinfty}:
\begin{align*}
\int_{0}^{t} \bar{z}(\theta)^T w(t-\theta)d\theta = \int_{0}^t \bar{w}(\theta)^Tz(t-\theta)d\theta.
\end{align*}
If one were to define a Laplace transform for these inputs ($\hat{w},\hat z,\hat{\bar w},\hat{\bar z}$) and transfer function for the systems ($\hat z(s)=G(s)\hat w(s)$ and $\hat{\bar z}(s)=G_d(s)\hat{\bar w}(s)$), then this equation would imply $\hat{\bar z}(s)^T\hat{w}(s)=\hat{\bar w}(s)^T \hat z(s)$ or $\hat{\bar w}(s)^T G_d(s)^T\hat{w}(s)=\hat{\bar w}(s)^T G(s)\hat w(s)$ so that $G_d(s)^T=G(s)$ ---  which is precisely the standard interpretation of the dual transfer function for ODEs. In addition, we note that while Theorem~\ref{thm:dual_hinfty} considers input-output stability of the primal and dual, the relationship in Eq.~\eqref{eq:dual_io} holds for any finite time, $t$, and hence does not require the primal or dual to be input-output stable.
\end{rem}

\begin{rem} Finally, we remark that the significance of Theorems \ref{thm:dual_stable} and \ref{thm:dual_hinfty} is not simply that a dual representation exists but that it has the same parameterization as the primal (making the primal and dual interchangeable). In addition, the proofs of Theorems \ref{thm:dual_stable} and \ref{thm:dual_hinfty} do not utilize the algebraic structure of the PI algebra -- implying that the duality result (and intertwining relationship) holds for any class of well-posed systems parameterized by a set of bounded linear operators on a reflexive Hilbert space which is closed under adjoint.
\end{rem}
\vspace{-4mm}

\section{Primal-Dual LPIs for Stability and $\hinf$-Norm}\label{sec:analysis}
In this section, we propose equivalent primal and dual conditions for both exponential stability and $\hinf$-norm (defined as maximum $L_2$-gain). These conditions are generalizations of the primal-dual Lyapunov and KYP LMIs used for stability and $\hinf$-norm analysis of state-space ODEs -- e.g. $A^TP+PA<0$ and $AP+PA^T<0$ for stability. In Section~\ref{sec:noboundary_control}, the dual conditions will be used to synthesize stabilizing and optimal controllers.

The primal-dual conditions in this section and Section~\ref{sec:noboundary_control} are formulated using Linear PI operator inequality (LPI) conditions. Such LPI constraints are a generalization of LMI constraints -- e.g. $\mcl A^*\mcl P\mcl T+\mcl T^*\mcl P\mcl A\prec 0$ and $\mcl A\mcl P\mcl T^*+\mcl T \mcl P\mcl A^*\prec 0$ for stability. Such conditions are constructed using a generalization of the quadratic storage functions used for state-space ODEs -- i.e. $V(x)=x^TPx$ becomes $V(\mbf x)=\ip{\mcl T\mbf x}{\mcl P\mcl T\mbf x}$. Note, however, that the duality results presented here are not quite as direct as in the state-space ODE case -- where if $V=x^TPx$ proves the stability of $\dot x=Ax$, then $V(x)=x^TP^{-1}x$ proves the stability of the dual.

Because the conditions presented here are in the form of LPIs, we implicitly assume that these conditions can be solved efficiently using convex optimization. Fortunately, computational methods for solving LPIs are well-developed and have been presented in other work~\cite{shivakumar_2019CDC}. In brief, these methods construct a positive PI operator using a quadratic form involving a positive matrix and \textit{$n^{th}$-order basis} of PI operators, $\mcl Z_n$. For example, $\mcl P \succeq 0$ if there exists some matrix $Q\ge0$ such that $\mcl P=\mcl Z_n^* Q\mcl Z_n=\mcl Z_n^* Q^{\half}Q^{\half} \mcl Z_n\succeq 0$, where the basis $\mcl Z_n$ is constructed using the vector of monomials, $Z_n$, of degree $n$ or less as
\begin{equation}\label{eq:monomial_bases}
\mcl Z_n\bmat{x\\\mbf x}(s) := \bmat{x\\Z_n(s)\mbf x(s)\\\int_a^s (Z_n(s)\otimes Z_n(\theta))\mbf x(\theta)d\theta\\\int_s^b (Z_n(s)\otimes Z_n(\theta))\mbf x(\theta)d\theta}.
\end{equation}
The highest order of these monomials, $n$, can be used as a measure for the complexity of the LPIs, which will later be used in Section~\ref{sec:numerical} to numerically verify the accuracy and convergence of the conditions.
For this numerical verification, the associated LPI conditions will be enforced using the Matlab toolbox, PIETOOLS. This toolbox offers convenient functions to convert PDEs to PIE, declare PI decision variables, add LPI constraints, and solve the resulting optimization problem. We refer to the PIETOOLS User Manual~\cite{manual} and \cite{das2020digital} for details.

\subsection{Primal and Dual LPIs for Stability of PDEs}\label{subsec:LPI_stab}
In the following theorem, we propose primal and dual LPI tests for exponential stability and use Cor.~\ref{cor:exponential_stable} to show that feasibility of either implies exponential stability of both the primal and dual systems.

\begin{thm}\label{thm:stable_LPI} Suppose that \textbf{either} of the two statements hold for some $\alpha>0$ and $\mcl P\in \PI_4$ such that $\mcl P=\mcl P^* \succeq \eta I$ for some $\eta>0$.
\begin{enumerate}[label=\alph*)]
\item $\mcl{T}^*\mcl{P}\mcl{A}+\mcl{A}^*\mcl{P}\mcl{T}\preceq -2\alpha \mcl T^*\mcl P\mcl T$
\item $\mcl{T}\mcl{P}\mcl{A}^*+\mcl{A}\mcl{P}\mcl{T}^*\preceq -2\alpha \mcl T\mcl P\mcl T^*$
\end{enumerate}
Then the PIEs defined by $\{\mcl T, \mcl A\} \subset \PI_4$ and  $\{\mcl T^*, \mcl A^*\} \subset \PI_4$ are \textbf{Exponentially Stable} with decay rate $\alpha$.
\end{thm}
\begin{proof}
Suppose a) holds. Define $V(\mbf x) = \ip{\mcl T\mbf x}{\mcl P\mcl T\mbf x}_{\R L_2}$. Since any $\mcl P\in\mbf \Pi_4$ is bounded in $\mcl L(\R L_2)$, 
\[
\eta \norm{\mcl T\mbf x}^2\le V(\mbf x)\le \norm{\mcl P}_{\mcl L(\R L_2)}\norm{\mcl T\mbf x}^2.
\]
Suppose $\mbf x(t)$ satisfies $\mbf x(0)=\mbf x_0$ and $\mcl T \dot{\mbf x}(t)=\mcl A \mbf x(t)$. Differentiating $V(\mbf x(t))$ with respect to time, $t$, we obtain
\begin{align*}
\dot{V}(\mbf x(t)) &= \ip{\mcl T\mbf x(t)}{\mcl P\mcl A\mbf x(t)}+\ip{\mcl A\mbf x(t)}{\mcl P\mcl T\mbf x(t)}\\
&= \ip{\mbf x(t)}{(\mcl T^*\mcl P\mcl A+\mcl A^*\mcl P\mcl T)\mbf x(t)} \le -2\alpha V(\mbf x(t)).
\end{align*}
Therefore, $\dot{V}(\mbf x(t))\le -2\alpha V(\mbf x(t))$ for all $t$ and, from the Gronwall-Bellman inequality, $V(\mbf x(t))\le V(\mbf x(0))e^{-2\alpha t}$. Let $\beta=\norm{\mcl T}_{\mcl L(\R L_2)}$ and $\zeta = \norm{\mcl P}_{\mcl L(\R L_2)}$. Then
\begin{align*}
\norm{\mcl T\mbf x(t)}^2 &\le \frac{1}{\eta}V(\mbf x(t)) \le \frac{1}{\eta}V(\mbf x(0))e^{-2\alpha t}\\
&\le \frac{1}{\eta}\zeta\norm{\mcl T\mbf x(0)}^2e^{-2\alpha t}\le \frac{\zeta\beta^2}{\eta}\norm{\mbf x(0)}^2e^{-2\alpha t}.
\end{align*}
By taking square root on both sides, $\norm{\mcl T\mbf x(t)}\le M\norm{\mbf x(0)}e^{-\alpha t}$ where $M = \sqrt{\frac{\zeta}{\eta}}\beta$.
This implies the PIE defined by $\{\mcl T, \mcl A\} \subset \PI_4$ is \textbf{Exponentially Stable} with decay rate $\alpha$.
Then, from Corollary \ref{cor:exponential_stable}, the PIE defined by $\{\mcl T^*, \mcl A^*\} \subset \PI_4$ is \textbf{Exponentially Stable} with decay rate $\alpha$.

The proof similarly establishes exponential stability for b) by swapping $\mcl T\mapsto \mcl T^*$ and $\mcl A \mapsto \mcl A^*$ and proving stability of the dual system.
\end{proof}
Both a) and b) in Thm.~\ref{thm:stable_LPI} imply exponential stability of both primal and dual using the definition of exponential stability in Def.~\ref{defn:expstability}, $
\norm{\mcl T\mbf x(t)}_{\R L_2}\le M\norm{\mbf x_0}_{\R L_2}e^{-\alpha t}$ where the upper bound is defined using the $L_2$-norm of the PIE state (which is equivalent to the Sobolev norm of the PDE state). This slightly stronger norm is needed to preserve the symmetry of the primal and dual. However, we also note that from the proof of Thm.~\ref{thm:stable_LPI}, a) implies exponential stability of the primal and b) implies exponential stability of the dual using an upper bound of the form $\norm{\mcl T\mbf x(t)}_{\R L_2}\le M\norm{\mcl T \mbf x_0}_{\R L_2}e^{-\alpha t}$. Practically, however, there is no difference between these definitions of exponential stability since we always assume that $\mbf x_0 \in \R L_2$.

\subsection{Primal and Dual KYP Lemma for PDEs}\label{subsec:LPI_kyp}
In the following theorem, we propose LPI generalizations of the primal and dual versions of the KYP Lemma and use Theorem~\ref{thm:dual_hinfty} to show that the solution of either proves a bound on the $L_2$-gain of both the primal and dual systems.

Note that the LPI conditions in Theorem~\ref{thm:gain_lpi} are expressed using an extension of block matrices to block PI operators -- The formal definition of concatenation of PI operators can be found in~\cite[Lemmas 40 and 41]{shivakumar_representation_TAC}. However, because the domain and range of PI operators of the form given in Def.~\ref{def:4PI} are an ordered concatenation of $\R$ and $L_2$, the arrangement of the blocks of the operators in the proposed LPI conditions are slightly different from that in the traditional formulations of the KYP Lemma for state-space ODEs.

\begin{thm}\label{thm:gain_lpi}
Suppose that \textbf{either} of the two statements hold for some $\gamma>0$ and $\mcl P\in \PI_4$ such that $\mcl P=\mcl P^* \succeq 0$.
\begin{enumerate}[leftmargin=*,labelsep=2pt,label=\alph*)] \setlength{\abovedisplayskip}{0pt}
\item	$\bmat{-\gamma I& D& \mcl{C}\\D^T&-\gamma I & \mcl{B}^*\mcl P\mcl T\\\mcl{C}^*&\mcl{T}^*\mcl P\mcl B&\mcl{T}^*\mcl P\mcl{A}+\mcl{A}^*\mcl P \mcl{T}}\hspace{-1.0mm}\preceq\hspace{-0.5mm} 0$

\item	$\bmat{-\gamma I&D^T&\mcl B^*\\
D&-\gamma I&\mcl{C}\mcl{P}\mcl T^*\\
\mcl B&\mcl{T}\mcl{P}\mcl C^*&\mcl T\mcl{P}\mcl{A}^*+\mcl{A}\mcl{P}\mcl T^*}	\hspace{-1.0mm}\preceq\hspace{-0.5mm} 0$
\end{enumerate}
Then, for any $w\in L_2$, if $z$ satisfies either
\begin{equation}
\bmat{\partial_t(\mcl{T}\mbf{x}(t))\\z(t)}  = \bmat{\mcl{A}&\mcl{B}\\\mcl C& D}\bmat{\mbf x(t)\\w(t)},\label{eqn:primal_temp}
\end{equation}
or
\begin{equation}
\bmat{\mcl{T}^*\dot{\mbf{x}}(t)\\z(t)}  = \bmat{\mcl{A}^*&\mcl{C}^*\\\mcl B^*& D^T}\bmat{\mbf x(t)\\w(t)},\label{eqn:dual_temp}
\end{equation}
for some $\mbf x(t)$ with $\mcl T\mbf{x}(0)=0$, then $\Vert z\Vert_{L_2} \le \gamma\norm{w}_{L_2}$.

\end{thm}
\begin{proof}
Suppose a) holds. Define $V(\mbf x) = \ip{\mcl{T}\mbf{x}}{\mcl P \mcl{T}\mbf{x}}_{\R L_2}$. For any $w\in L_2$, suppose $z$ satisfies Eq.~\eqref{eqn:primal_temp} for some $\mbf x$ with $\mbf{x}(0)=0$. Differentiating $V(\mbf x(t))$ with respect to time, $t$, we obtain
\begin{align*}
&\dot{V}(\mbf x(t)) = \ip{\mcl{T}\mbf x(t)}{\mcl P \left(\mcl A\mbf x(t) + \mcl Bw(t)\right)}\\
&\qquad \qquad+ \ip{\left(\mcl A\mbf x(t) + \mcl Bw(t)\right)}{\mcl P \mcl{T}\mbf x(t)}\\
&= \ip{\bmat{w(t)\\\mbf x(t)}}{\bmat{0&\mcl{B}^*\mcl P\mcl T\\\mcl{T}^*\mcl P\mcl B&\mcl{T}^*\mcl P\mcl{A}+\mcl{A}^*\mcl P \mcl{T}}\bmat{w(t)\\\mbf{x}(t)}}.
\end{align*}
Now defining auxiliary variable $v(t) = \frac{1}{\gamma}z(t)$, we obtain

\begin{align*}
&\dot{V}(\mbf x(t))- \gamma\norm{w(t)}_{\R}^2 +\frac{1}{\gamma}\norm{z(t)}_{\R}^2\\
&=\dot{V}(\mbf x(t))- \gamma \norm{w(t)}^2-\frac{1}{\gamma}\norm{z(t)}^2+\frac{2}{\gamma}\norm{z(t)}^2\\
&= \dot{V}(\mbf x(t)) -\gamma\norm{w(t)}^2 -\gamma \norm{v(t)}^2+v(t)^T z(t)+z(t)^T v(t)\\
&={\small\ip{\bmat{v(t)\\w(t)\\\mbf x(t)}}{\bmat{-\gamma I& D& \mcl{C}\\D^T&-\gamma I & \mcl{B}^*\mcl P\mcl T\\\mcl{C}^*&\mcl{T}^*\mcl P\mcl B&\mcl{T}^*\mcl P\mcl{A}+\mcl{A}^*\mcl P \mcl{T}}\bmat{v(t)\\w(t)\\\mbf x(t)}}}\le 0.
\end{align*}
Integrating this inequality in time, we obtain
\begin{align*}
V(\mbf x(T))-V(\mbf x(0))&\le  \gamma\int_0^T\norm{w(t)}^2 dt  -\frac{1}{\gamma}\int_0^T\norm{z(t)}^2 dt.
\end{align*}
Now, since $\mcl T\mbf x(0)=0$ and $V(\mbf x(T))\ge 0$ for all $T\ge 0$, we obtain $\norm{z}^2_{L_2}\le\gamma^2\norm{w}^2_{L_2}$. Furthermore, Theorem~\ref{thm:dual_hinfty} implies the same bound hold if $z$ and $\mbf x$ satisfy Eq.~\eqref{eqn:dual_temp}.

Since $\mcl T^{**}=\mcl T$ and
\[
\bmat{(\mcl{A}^{*})^*&(\mcl{B}^{*})^*\\(\mcl C^{*})^*& (D^{T})^T}=\bmat{\mcl{A}^*&\mcl{C}^*\\\mcl B^*& D^T}^*=\bmat{\mcl{A}&\mcl{B}\\\mcl C& D}^{**}=\bmat{\mcl{A}&\mcl{B}\\\mcl C& D},
\]
we have that b) likewise implies the same bounds.
\end{proof}

Before applying the results of Theorem~\ref{thm:gain_lpi} to controller synthesis, we note that the operator variable $\mcl P$,  which represents the storage function $V(\mbf x)=\ip{\mcl T\mbf x}{\mcl P\mcl T \mbf x}$ in Theorem~\ref{thm:gain_lpi}, is not required to be \textit{strictly} positive -- thus allowing for the use of non-coercive storage functions (See~\cite{jacob_2018}). When we turn to the problem of optimal controller synthesis in Subsection~\ref{subsec:optimalcontrol}, however, we will require strict positivity of $\mcl P$ so that we may reconstruct the controller gains as $\mcl K=\mcl Z\mcl P^{-1}$.

\subsection{Duality using Extended Lyapunov Functions}\label{subsec:alternativeLPI}
The LPI conditions for exponential stability and $L_2$-gain in Thms.~\ref{thm:stable_LPI} and~\ref{thm:gain_lpi} were obtained by parameterizing candidate Lyapunov/storage functions of the form $V(\mbf x)=\ip{\mcl T\mbf x}{\mcl P \mcl T\mbf x}_{L_2}$ where $\mbf x$ is the PIE state. This parametrization was chosen to ensure that the function $V$ is both lower and upper bounded with respect to the original PDE state, $\mbf v:=\mcl T \mbf x$ -- i.e. $V=\ip{\mcl T\mbf x}{\mcl P \mcl T\mbf x}_{L_2}=\ip{\mbf v}{\mcl P \mbf v}_{L_2}$ and $\mcl P \succ 0$ implies $\alpha \norm{\mbf v}^2\le V(\mbf v)\le \beta \norm{\mbf v}^2$ for some $\alpha,\beta>0$. However, these upper and lower bounds are unnecessary when computing $L_2$-gain. Therefore, in this subsection we propose an extension of Thm.~\ref{thm:gain_lpi} which allows $V$ to be defined by a mix of PIE and PDE states as $V(\mbf x)=\ip{\mcl Q\mbf x}{\mcl T\mbf x}_{L_2}$ for some PI operator $Q$. By including the PIE state, $\mbf x$, in addition to the PDE state, $\mcl T\mbf x$, this \textit{extended} functional form allows for partial derivatives of the PDE state to appear. 

To illustrate, consider the heat equation $\partial_t{\mbf v}(t) = \partial_s^2 \mbf v(t)$ with boundary conditions $\mbf v(0)=\partial_s\mbf v(1)=0$. 
The extended parametrization now allows us to use storage functions defined in terms of partial derivatives -- such as $V=\norm{\mbf v_s}^2$. Specifically, if we choose $\mcl Q=-I$, and defining the PIE state as $\mbf x := \partial_s^2\mbf v$ (where $\mbf v=\mcl T\mbf x$ for some $\mcl T$), we have 
\[
V(\mbf x)=\ip{\mcl Q\mbf x}{ \mcl T\mbf x}=-\ip{\partial_s^2\mbf v}{\mbf v}=\norm{\partial_s\mbf v}^2.
\]
%

We note that this extended parametrization of storage functions includes, as a subset, those functions used in Thm.~\ref{thm:gain_lpi}. Specifically, for any given $\mcl P$, if we choose $\mcl Q$ as $\mcl Q=\mcl P \mcl T$, then the extended storage function has the form $V(\mbf x)=\ip{\mcl Q\mbf x}{\mcl T\mbf x}_{L_2}=\ip{\mcl T\mbf x}{\mcl P \mcl T\mbf x}_{L_2}$, as was used in Thm.~\ref{thm:gain_lpi}.
The drawback of such extended Lyapunov function candidates, however, is that they are not upper-bounded with respect to the PDE state since the $L_2$ norms and Sobolev norms are not equivalent.
This means that we cannot extend the exponential stability criterion in Thm.~\ref{thm:stable_LPI} without redefining our notion of exponential stability.

\begin{thm}\label{thm:gain_lpi_new}
Suppose that \textbf{either} of the two statements hold for some $\gamma>0$ and $\mcl R, \mcl Q\in \PI_4$ such that $\mcl R=\mcl R^* \succeq 0$.
\begin{enumerate}[leftmargin=*,labelsep=2pt,label=\alph*)] \setlength{\abovedisplayskip}{0pt}
\item $\mcl T^*\mcl Q=\mcl Q^*\mcl T=\mcl R$ and \[ 
\bmat{-\gamma I&D&\mcl C\\
D^T&-\gamma I&\mcl{B}^*\mcl{Q}\\
\mcl C^*&\mcl{Q}^*\mcl B&\mcl Q^*\mcl{A}+\mcl{A}^*\mcl Q}  \hspace{-1.0mm}\preceq\hspace{-0.5mm} 0\]

\item $\mcl T\mcl Q^*=\mcl Q\mcl T^*=\mcl R$ and \[ 
\bmat{-\gamma I&D^T&\mcl B^*\\
D&-\gamma I&\mcl{C}\mcl{Q}^*\\
\mcl B&\mcl{Q}\mcl C^*&\mcl Q\mcl{A}^*+\mcl{A}\mcl Q^*}  \hspace{-1.0mm}\preceq\hspace{-0.5mm} 0\]
\end{enumerate}
Then, for any $w\in L_2$, if $z$ satisfies either
\begin{equation}
\bmat{\partial_t(\mcl{T}\mbf{x}(t))\\z(t)}  = \bmat{\mcl{A}&\mcl{B}\\\mcl C& D}\bmat{\mbf x(t)\\w(t)},\label{eqn:primal_temp_2}
\end{equation}
or
\begin{equation}
\bmat{\mcl{T}^*\dot{\mbf{x}}(t)\\z(t)}  = \bmat{\mcl{A}^*&\mcl{C}^*\\\mcl B^*& D^T}\bmat{\mbf x(t)\\w(t)},\label{eqn:dual_temp_2}
\end{equation}
for some $\mbf x(t)$ with $\mbf{x}(0)=0$, then $\Vert z\Vert_{L_2} \le \gamma\norm{w}_{L_2}$.
\end{thm}
\begin{proof}
Suppose a) holds.  Define $V(\mbf x)=\ip{\mbf x}{\mcl R \mbf x}\ge 0$. Since $\mcl Q^*\mcl T = \mcl T^*\mcl Q$, we have $V(\mbf x)=\ip{\mcl T\mbf x}{\mcl Q \mbf x}=\ip{\mcl Q\mbf x}{\mcl T \mbf x}$.
Now, for any $w\in L_2$, suppose $z$ satisfies Eq.~\eqref{eqn:primal_temp_2} for some $\mbf x$ with $\mbf{x}(0)=0$. Differentiating $V(\mbf x(t))$ with respect to time, $t$, we obtain
\begin{align*}
&\dot{V}(\mbf x(t)) = \ip{\partial_t(\mcl T\mbf x(t))}{\mcl Q\mbf x(t)}_{\R L_2}+\ip{\mcl Q\mbf x(t)}{\partial_t(\mcl T\mbf x(t))}_{\R L_2}\\
&= \ip{\mcl Q\mbf x(t)}{\mcl A\mbf x(t) + \mcl Bw(t)}+ \ip{\mcl A\mbf x(t) + \mcl Bw(t)}{\mcl Q\mbf x(t)}\\
&= \ip{\bmat{w(t)\\\mbf x(t)}}{\bmat{0&\mcl{B}^*\mcl{Q}\\\mcl{Q}^*\mcl B&\mcl{Q}^*\mcl{A}+\mcl{A}^*\mcl Q }\bmat{w(t)\\\mbf{x}(t)}}.
\end{align*}
Now defining auxiliary variable $v(t) = \frac{1}{\gamma}z(t)$, we obtain
\begin{align*}
&\dot{V}(\mbf x(t))- \gamma\norm{w(t)}_{\R}^2 +\frac{1}{\gamma}\norm{z(t)}_{\R}^2\\
&=\dot{V}(\mbf x(t))- \gamma \norm{w(t)}^2-\frac{1}{\gamma}\norm{z(t)}^2+\frac{2}{\gamma}\norm{z(t)}^2\\
&= \dot{V}(\mbf x(t)) -\gamma\norm{w(t)}^2 -\gamma \norm{v(t)}^2+v(t)^T z(t)+z(t)^T v(t)\\
&={\small\ip{\bmat{v(t)\\w(t)\\\mbf x(t)}}{\bmat{-\gamma I&D&\mcl C\\
D^T&-\gamma I&\mcl{B}^*\mcl{Q}\\
\mcl C^*&\mcl{Q}^*\mcl B&\mcl Q^*\mcl{A}+\mcl{A}^*\mcl Q}\bmat{v(t)\\w(t)\\\mbf x(t)}}}\le 0.
\end{align*}
The rest of the proof is exactly as in the proof of Thm.~\ref{thm:gain_lpi}.

\end{proof}
Finally, we note that feasibility of the conditions in Thm.~\ref{thm:gain_lpi} imply the conditions of Thm.~\ref{thm:gain_lpi_new} are satisfied with $\mcl Q=\mcl P \mcl T$ for a) and   $\mcl Q=\mcl P \mcl T^*$ for b). In addition, if it is known that if $\{\mcl T, \mcl A\}$ or $\{\mcl T^*, \mcl A^*\}$ is asymptotically stable then the non-negativity constraints on $\mcl P$ in Thm.~\ref{thm:gain_lpi} and on $\mcl Q$ in Thm.~\ref{thm:gain_lpi_new} may be removed entirely.

\section{LPIs for Stabilizing and $H_\infty$-Optimal State-Feedback Controller Synthesis}\label{sec:noboundary_control}
In this section, we return to the state-feedback controller synthesis problems defined in Section~\ref{sec:PIEs} (Eqs.~\eqref{eq:stabilization_PIE} and~\eqref{eq:hinf_opt_PIE}). Specifically, given a PIE system
\[
\bmat{\partial_t(\mcl{T}\mbf{x}(t))\\z(t)}  = \bmat{\mcl{A}&\mcl{B}_1&\mcl{B}_2\\\mcl C& D_1 &D_2}\bmat{\mbf x(t)\\w(t)\\u(t)},
\]
our goal is to synthesize state-feedback controllers of the form $u(t)=\mcl K \mbf x(t)$, where $\mbf x$ is the state of the PIE and the controller gain, $\mcl K$, is a PI operator.
To do this, we apply Corollary~\ref{cor:exponential_stable} and Theorem ~\ref{thm:gain_lpi} to the closed-loop system
\begin{align}\label{eq:PIE_cl}
\bmat{\partial_t(\mcl{T}\mbf{x}(t))\\z(t)}  = \bmat{\mcl{A} +\mcl{B}_2 \mcl K&\mcl{B}_1\\\mcl C+D_2\mcl K& D_1}\bmat{\mbf x(t)\\w(t)}.
\end{align}
The resulting operator inequality then includes the term $\mcl K \mcl P$ which is bilinear in the decision variables $\mcl K$ and $\mcl P$. However, as described in the introduction, and following the approach used for state-space ODEs, we then construct an equivalent LPI by making the invertible variable substitution $\mcl K \mcl P \rightarrow \mcl Z$. An iterative algorithm for the inversion of this variable substitution is presented in Section~\ref{sec:controller} -- allowing us to reconstruct the controller gains for implementation in simulation or real-time feedback. 

\subsection{Stabilizing State-Feedback Control}\label{subsec:stablecontrol}

The following Corollary defines an LPI whose solution provides an exponentially stabilizing state-feedback controller for the PDE associated with the PIE $\{\mcl T,\mcl A,\mcl B_2\}$.

\begin{cor}\label{cor:stabilizing_control_lpi}
Suppose there exist some $\alpha>0$ and $\mcl Z, \mcl P\in \PI_4$ such that $\mcl P=\mcl P^* \succeq \eta I$ for some $\eta>0$, and\vspace{1mm}
\[ (\mcl{AP+B}_2\mcl Z)\mcl{T}^*+\mcl{T}(\mcl{AP+B}_2\mcl Z)^*\preceq -2\alpha \mcl{T}\mcl P\mcl{T}^*.\vspace{1mm}
\]
Then if $\mcl K= \mcl Z \mcl P^{-1}$, the PIE defined by $\{\mcl T, \mcl A+\mcl B_2 \mcl K\} \subset \PI_4$ is \textbf{Exponentially Stable} with decay rate $\alpha$.
\end{cor}
\begin{proof}
Let $\mcl P$, $\mcl Z$, and $\mcl K$ be as defined above. Then, $\mcl Z = \mcl K \mcl P$, and
\begin{align*}
&(\mcl{AP+B}_2\mcl Z)\mcl{T}^*+\mcl{T}(\mcl{AP+B}_2\mcl Z)^* \\
&\qquad= (\mcl{AP}+\mcl B_2\mcl{KP})\mcl{T}^*+\mcl{T}(\mcl{AP}+\mcl B_2\mcl{KP})^*\\
&\qquad = (\mcl{A+B}_2\mcl K)\mcl P\mcl{T}^*+\mcl{T}\mcl P(\mcl{A+B}_2\mcl K)^*\preceq -2\alpha \mcl{T}\mcl P\mcl{T}^*.
\end{align*}

Then, from Theorem \ref{thm:stable_LPI} (statement b), the PIE defined by $\{\mcl{T},\mcl{A+B}_2\mcl K\}$, is exponentially stable with decay rate $\alpha$.
\end{proof}

\subsection{$\hinf$-Optimal State-Feedback Control}\label{subsec:optimalcontrol}

Next, we provide an LPI to find the $\hinf$-optimal state-feedback controller, $\mcl K$, for PIEs with inputs and outputs of the form Eq.~\eqref{eq:PIE_cl}. Here, we use $(\cdot)^*$ notation to represent the symmetric adjoint/transpose completion of block operators.

\begin{cor}\label{cor:optimal_control_lpi}
Suppose there exist some $\gamma>0$ and $\mcl Z, \mcl P\in \PI_4$ such that $\mcl P=\mcl P^* \succeq \eta I$ for some $\eta>0$, and
\[\bmat{-\gamma I&D_{1}^T&\mcl B_1^*\\
(\cdot)^*&-\gamma I&(\mcl{C}\mcl{P}+D_{2}\mcl{Z})\mcl{T}^*\\
(\cdot)^*&(\cdot)^*&(\cdot)^*+\mcl T(\mcl{A}\mcl{P}+\mcl{B}_2\mcl{Z})^*}\preceq 0.\]
Then if $\mcl K= \mcl Z \mcl P^{-1}$,  for any $w\in L_2$, if $z$ satisfies 	
\[
\bmat{\partial_t(\mcl{T}\mbf{x}(t))\\z(t)}  = \bmat{\mcl{A} +\mcl{B}_2 \mcl K&\mcl{B}_1\\\mcl C+D_2\mcl K& D_1}\bmat{\mbf x(t)\\w(t)},
\]
for some $\mbf x$ with $\mcl T\mbf x(0)=0$, then $\norm{z}_{L_2}\le \gamma \norm{w}_{L_2}$.
\end{cor}
\begin{proof}
Let $\mcl P$, $\mcl Z$, and $\mcl K$ satisfy the corollary statement. Then, $\mcl Z = \mcl K \mcl P$, and
\begin{align*}
&\bmat{-\gamma I&D_{1}^T&\mcl B_1^*\\
(\cdot)^*&-\gamma I&(\mcl{C}\mcl{P}+D_{2}\mcl{Z})\mcl{T}^*\\
(\cdot)^*&(\cdot)^*&(\cdot)^*+\mcl T(\mcl{A}\mcl{P}+\mcl{B}_2\mcl{Z})^*}\\
&= \bmat{-\gamma I&D_{1}^T&\mcl B_1^*\\
(\cdot)^*&-\gamma I&(\mcl{C}+D_{2}\mcl K)\mcl P\mcl{T}^*\\
(\cdot)^*&(\cdot)^*&(\cdot)^*+\mcl T\mcl{P}(\mcl{A}+\mcl B_2\mcl K)^*}\preceq0.
\end{align*}

Thus, from Theorem \ref{thm:gain_lpi} (statement b),  for $z,w,\mbf x$ as in the corollary statement, we have that $\Vert z\Vert_{L_2} \le \gamma\norm{w}_{L_2}$.
\end{proof}

Given a PDE with associated PIE defined by $\{\mcl T, \mcl A,\mcl B_i,\mcl C, D_i\}$, Corollary \ref{cor:optimal_control_lpi} provides a controller gain $\mcl K=\mcl Z\mcl P^{-1}$  such that $u(t)=\mcl K \mbf x(t)$  achieves a closed-loop performance bound of $\norm{z}_{L_2}\le \gamma \norm{w}_{L_2}$. Note that this controller does not necessarily imply internal exponential stability unless the LPI in Corollary~\ref{cor:stabilizing_control_lpi} is negative definite in a suitable sense.

\subsection{Optimal Controllers using Extended Lyapunov Functions}\label{subsec:alternativeLPI_control}
In this subsection, we briefly propose alternative optimal state-feedback controller synthesis conditions based on the extended $L_2$-gain conditions described in Subsection~\ref{subsec:alternativeLPI} and which allow for Lyapunov/storage functions which include partial derivatives of the state.  Specifically, for the problem of $H_\infty$-optimal state-feedback controller synthesis, we have the following alternative formulation of Thm.~\ref{thm:gain_lpi_new}.
\begin{cor}\label{cor:optimal_control_LPI_new}
Suppose there exist some $\gamma>0$ and $\mcl Z, \mcl R, \mcl Q\in \PI_4$ such that $\mcl T\mcl Q=\mcl Q^*\mcl T^*=\mcl R\succeq 0$, $\mcl Q$ is invertible, and
\[\bmat{-\gamma I&D_{1}^T&\mcl B_1^*\\
D_1&-\gamma I&\mcl{C}\mcl{Q}+D_{2}\mcl{Z}\\
\mcl B_1&\hspace{-2mm}\mcl{Q}^*\mcl{C}^*+\mcl{Z}^*D_{2}^T&(\mcl{A}\mcl{Q}+\mcl{B}_2\mcl{Z})^*+(\mcl{A}\mcl{Q}+\mcl{B}_2\mcl{Z})}\preceq 0.\]
Then if $\mcl K= \mcl Z \mcl Q^{-1}$,  for any $w\in L_2$, if $z$ satisfies   
\[
\bmat{\partial_t(\mcl{T}\mbf{x}(t))\\z(t)}  = \bmat{\mcl{A} +\mcl{B}_2 \mcl K&\mcl{B}_1\\\mcl C+D_2\mcl K& D_1}\bmat{\mbf x(t)\\w(t)},
\]
for some $\mbf x$ with $\mcl T\mbf x(0)=0$, then $\norm{z}_{L_2}\le \gamma \norm{w}_{L_2}$.
\end{cor}
\begin{proof}
The proof is similar to the proof of Cor.~\ref{cor:optimal_control_lpi} using the alternative conditions defined in Thm.~\ref{thm:gain_lpi_new}.
\end{proof}

\section{Controller Reconstruction}\label{sec:controller}

In this section, we presume that for some given PIE, the LPIs in Corollary~\ref{cor:stabilizing_control_lpi},~\ref{cor:optimal_control_lpi}, or~\ref{cor:optimal_control_LPI_new} have solution $\mcl Z,\mcl P \in \mbf \Pi_4$ where $\mcl P \succ 0$. In this case, a controller for the PDE (with state $\mbf v$) associated with the given PIE may be constructed as $u(t)=\mcl Z \mcl P^{-1}\mbf x(t)$ where $\mbf x$ is the PIE state. Now suppose $\mcl K=\mcl Z \mcl P^{-1}$ is a PI operator, and $\mbf x(t)\in \R L_2$ is partitioned as $\mbf x(t,s)=[x_1(t),\mbf x_2(t,s)]$. Then such a controller has the form \vspace{-2mm}
\[
u(t)=K_1 x_1(t)+\int_a^b K_2(s)\mbf x_2(t,s)ds.\vspace{-1mm}
\]
Real-time estimates of the $\mbf x_2(t)$ may be obtained by measurement and spatial differentiation of the PDE state or through use of an estimator, as described in~\cite{das_2019CDC}. However, construction of the gains $K_1,K_2(s)$ requires us to compute the inverse operator $\mcl P^{-1}$.

The question of inverting a PI operator of the form $\mcl P=\fourpi{P}{Q}{Q^T}{R_i}\in \PI_4$ has been considered in~\cite{peet_sicon2020} under the restriction $R_1=R_2$ (the case of \textit{separable} kernels) and without the terms $P,Q$ (no ODE states). Unfortunately, the restriction $R_1=R_2$ often results in suboptimal controllers. In this section, therefore, we lift the restriction $R_1=R_2$ using a new method for inverting PI operators based on generalization of a result in~\cite[Chapter IX.2]{gohberg}. This method is presented in stages: first neglecting ODE states and restricting $R_0=I$ (Lemma~\ref{lem:inverse_gohberg}); then accounting for $R_0$ (Cor.~\ref{cor:inverse_gohberg}); then accounting for ODE states (Lemma~\ref{lem:inverse_b}). 

Because the initial steps of this method do not consider ODE states, Lemma~\ref{lem:inverse_gohberg} and Cor.~\ref{cor:inverse_gohberg} utilize the simplified notation
\[
\threepi{R_0,R_1,R_2}:=\fourpi{\emptyset}{\emptyset}{\emptyset}{R_0,R_1,R_2}.
\]
The following lemma provides a construction for the inverse of a PI operator, $\threepi{I,H_1,H_2}$ where separability of parameters $H_1,H_2$ is implied by separability of polynomials --  i.e. any polynomial, $H$ can be written as $H(s,\theta) = F(s)G(\theta)$ for some polynomials $F, G$.   
\begin{lem}[Gohberg~\cite{gohberg}] \label{lem:inverse_gohberg} Define $\mcl P := \threepi{I,H_1,H_2}$ where $H_i(s,\theta) = -F_i(s)G_i(\theta)$ for some
$F_i, G_i\in L_2[a,b]$. Let $U$ and $V$ be the unique solutions to
\begin{align*}
    U(s) &= I +\int_a^s B(\theta)C(\theta)U(\theta) d\theta,\\
    V(\theta) &= I - \int_a^\theta V(s)B(s)C(s)ds,    
\end{align*}
where $C(s) = \bmat{F_1(s)&F_2(s)}$, $B(s) = \bmat{G_1(s)\\-G_2(s)}$. Then $V(s)U(s)=U(s)V(s)=I$. Furthermore, if we partition
\[
U(b) = \bmat{U_{11}&U_{12}\\U_{21}&U_{22}}, \quad U_{22}\in \R^{q\times q},
\]
where $q$ is the number of columns in $F_2$, then $\mcl P$ is invertible if and only if $U_{22}$ is invertible and $\mcl P^{-1}=\threepi{I,M_1,M_2}$ where
\[
P = \bmat{0&\hspace{-1.25mm}0\\U_{22}^{-1}U_{21}&\hspace{-1.25mm}I},~\mat{	M_1(s,\theta) = C(s)U(s)(I-P)V(\theta)B(\theta),\\M_2(s,\theta) = -C(s)U(s)PV(\theta)B(\theta).}
\]
\end{lem}
Note that, by construction,  $M_1$ and $M_2$ are \textit{separable} --  implying that $\mcl P^{-1}$ is a PI operator, albeit not necessarily with polynomial parameters. 

We now extend Lemma~\ref{lem:inverse_gohberg} to PI operators of the form $\mcl P=\threepi{R_0,R_1,R_2}$, where $\mcl P\succ 0$ implies invertibility of $R_0$.
\begin{cor}\label{cor:inverse_gohberg}
Suppose $R_0^{-1} \in L_2$ and define $H_1(s,\theta)=R_0(s)^{-1}R_1(s,\theta)$ and $H_2(s,\theta)=R_0(s)^{-1}R_2(s,\theta)$. Now let $M_1,M_2$ be as defined in Lemma~\ref{lem:inverse_gohberg}. Then $\threepi{R_i}^{-1}=\threepi{\hat R_i}$ where $\hat R_0=R_0^{-1}$, $\hat R_1(s,\theta)=M_2(s,\theta)R_0^{-1}(\theta)$, and $\hat R_2(s,\theta)=M_1(s,\theta)R_0^{-1}(\theta)$.
\end{cor}
\begin{proof}
The proof follows immediately from Lemma~\ref{lem:inverse_gohberg} and operator composition rules~\cite{shivakumar_representation_TAC} as
\begin{align*}
\threepi{R_0,R_1,R_2}^{-1}
&= (\threepi{R_0,0,0} \threepi{I,H_1,H_2})^{-1} \\
&= \threepi{I,M_1,M_2} \threepi{R_0^{-1},0,0}= \threepi{\hat R_0,\hat R_1 ,\hat R_2}.
\end{align*}
\end{proof}
Next, we extend Corollary~\ref{cor:inverse_gohberg} to $\mcl P\in \PI_4$ using a generalization of a standard formula for block matrix inversion.
\begin{lem}\label{lem:inverse_b}
Suppose $\threepi{R_i}^{-1}=\threepi{\hat R_i}$. Then $\mcl P:=\fourpi{P}{Q_1}{Q_2}{R_i} \in \PI_4$ is invertible if and only if the matrix
\[ T = P - \fourpi{\emptyset}{Q_1}{\emptyset}{\emptyset}\threepi{\hat R_i}\fourpi{\emptyset}{\emptyset}{Q_2}{\emptyset}\] is invertible. Furthermore, 
\[
\mcl P^{-1} = \mcl U\fourpi{T^{-1}}{0}{0}{\hat R_i}\mcl V
\]
where
\begin{align*}
&\mcl U = \fourpi{I}{0}{0}{\hat R_i}\fourpi{I}{0}{-Q_2}{R_i}\\
&\mcl V = \fourpi{I}{-Q_1}{0}{R_i}\fourpi{I}{0}{0}{\hat R_i}.
\end{align*}
\end{lem}

\begin{proof}
Let $\hat R$ be as specified and define
$\mcl R = \fourpi{I}{0}{0}{\hat R_i}$. Then $\mcl P$ can be decomposed as
\[
\mcl P= \overbrace{\fourpi{I}{Q_1}{0}{R_i}\mcl R}^{\mcl M:=}\overbrace{\fourpi{T}{0}{0}{R_i}}^{\mcl Q:=}\overbrace{\mcl R\fourpi{I}{0}{Q_2}{R_i}}^{\mcl N:=}.
\]
Clearly, $\mcl M,\mcl N$ are triangular, which implies $\mcl N^{-1} = \mcl U$ and $\mcl M^{-1} = \mcl V$. Hence invertibility of $\mcl P$ is now equivalent to the invertibility of  $\mcl Q$. Finally, we have \[
\mcl Q^{-1}=\fourpi{T}{0}{0}{R_i}^{-1}=\fourpi{T^{-1}}{0}{0}{\hat R_i}
\]
which completes the proof.
\end{proof}

Given the results in Lemma~\ref{lem:inverse_gohberg}, Cor.~\ref{cor:inverse_gohberg}, and Lemma~\ref{lem:inverse_b}, we now consider the numerical problem of computing the controller gains, $K_1$, $K_2$, that define
\[
\mcl K=\fourpi{K_1}{K_2}{\emptyset}{\emptyset}=\mcl Z \mcl P^{-1}
\]
where $\mcl Z$ and $\mcl P$ are obtained from Corollary~\ref{cor:stabilizing_control_lpi},~\ref{cor:optimal_control_lpi}, or~\ref{cor:optimal_control_LPI_new} and have the form
\[
\mcl Z=\fourpi{Z_1}{Z_2}{\emptyset}{\emptyset} \quad \text{and}\quad \mcl P:=\fourpi{P}{Q_1}{Q_2}{R_i}.
\]
Specifically, if $\mcl P^{-1}:=\fourpi{\hat P}{\hat Q_1}{\hat Q_2}{\hat R_i}$, then
\begin{align*}
K_1&=Z_1 \hat P +\int_a^b Z_2(s)\hat Q_2(s)ds\\
K_2(s)=&Z_1\hat Q_1(s)+Z_2(s)\hat R_0(s)\\
&+\int_a^sZ_2(\theta)\hat R_1(\theta,s)d\theta +\int_s^b Z_2(\theta)\hat R_2(\theta,s)d\theta.
\end{align*}
Of course, the parameters $\hat P,\hat Q_i,\hat R_i$ are obtained sequentially by first computing solutions $M_1,M_2$ to the Volterra-type integral equations (of $2^{nd}$ kind) in Lemma~\ref{lem:inverse_gohberg}, then applying the analytic formulae in Cor.~\ref{cor:inverse_gohberg} and Lemma~\ref{lem:inverse_b}. We note, however, that both $R_0^{-1}$ and the solutions $M_1,M_2$ will be non-polynomial and there are no closed-form analytic expressions for these parameters unless $R_1=R_2$. Fortunately, there exist convergent series expansions which can be used to approximate these terms arbitrarily well. For example,~\cite[Chapter IX.2]{gohberg} provides the following result. 

\begin{lem}\label{lem:volterra_properties}
Let $A:[a,b]\to \R^{n\times n}$ be Lebesgue integrable on $[a,b]$. Then, the series $I_n+\sum_{i=1}^{\infty} U_k(s)$, where $U_k = \int_a^s A(\theta)U_{k-1}(\theta)d\theta$ and $U_1(s) = \int_a^s A(\theta)d\theta$, converges uniformly on $s\in[a,b]$ to a unique function, $U:[a,b]\to \R^{n\times n}$, that solves	$U(s) = I_n +\int_a^s A(\theta)U(\theta)d\theta$.
Furthermore, for any $k\in \N$, \vspace{-2mm}
\begin{align*}
\norm{U_k(s)}\le \frac{1}{k!}\left(\int_a^b\norm{A(s)}ds\right)^k,\quad s\in[a,b].
\end{align*}
\end{lem}
As a practical matter, we typically only need to obtain values for the controller gains at discrete points in space, which also simplifies the problem of inversion of the matrix-valued function $R_0(s)$. A more detailed description of the computation of the inverse operator $\mcl P^{-1}$ as implemented in PIETOOLS is described in Appendix A of the supplemental document.

\section{Numerical Implementation and Verification}\label{sec:numerical}
In this section, we use the PIETOOLS software package to construct the PIE representation and solve the associated LPIs for analysis in Thms.~\ref{thm:stable_LPI}, \ref{thm:gain_lpi}, \ref{thm:gain_lpi_new}, and controller synthesis in Cor.~\ref{cor:optimal_control_lpi}. 
These tests are performed on several PDE models, including heat, wave, and beam equations. For cases of actuation at the boundary, an ODE input filter is added as discussed in Subsec.~\ref{subsec:note_boundary}. When analytic expressions for performance are available, comparison is made to these results.


In each case, the PI operators ($\mcl T,\mcl A$, et c.) which define the PIE representation of the given PDE are constructed using the PIETOOLS command line interface (See Chap. 4 of~\cite{manual}) which applies the conversion formulae in~\cite{shivakumar_representation_TAC}. The LPIs are constructed by declaration of operator-valued decision variables (See Chap. 5 of~\cite{manual}), algebraic manipulation of the operators (See Chap. 10 of~\cite{manual}), and enforcement of operator-valued inequalities (See Chap. 7 of~\cite{manual}). Operations for the construction of the feedback gains by operator inversion are described in Chaps. 7 and 10 of~\cite{manual}. This approach is used to either find a primal and dual lower bound on exponential decay rate (in Subsec.~\ref{subsec:num_exp}) or find an $\hinf$-optimal state-feedback controller (in Subsec.~\ref{subsec:num_control}). In the case of controller synthesis, the controller gains obtained from the LPI, $\mcl K$, are used to construct a closed-loop PDE for numerical simulation and verification of performance. Numerical simulation is performed using the PIESIM package described in~\cite{Peet_2024ACM} (See also Chap. 6 of~\cite{manual}). For reproducibility, similar demonstrations and numerical examples can be found in the PIETOOLS software package (See Chaps. 11 and 12 of~\cite{manual}). 


\subsection{Bounding Exponential Decay Rate (Thm.~\ref{thm:stable_LPI})}\label{subsec:num_exp}
We apply the primal and dual LPIs for the exponential decay rate in Thm.~\ref{thm:stable_LPI} to a linear delay-differential equation and a PDE reaction-diffusion equation to obtain the maximum lower bound on the exponential decay rate, $\alpha$. To maximize $\alpha$, we observe that the LPIs in Thm.~\ref{thm:stable_LPI} are convex in $\alpha$ for a fixed $\mcl P$ -- which implies a bisection search on $\alpha$ can be used to maximize the lower bound on exponential decay rate.

\begin{ex}[Exponential Stability of a Linear Time-Delay System]\label{ex:time-delay}
Consider the following autonomous linear delay-differential equation from, e.g.~\cite{mondie2005exponential}.
\begin{align*}
\dot{x}(t)= \bmat{-4&1\\0&-4}x(t)+\bmat{0.1&0\\4&0.1}x(t-0.5)
\end{align*}
Delay differential equations can be represented as a transport equation coupled to an ODE and formulae for conversion of a delay-differential equation to a PIE can be found in~\cite{peet_2021AUT}. For this example, both the primal and dual LPIs obtained maximum provable lower bounds on the exponential decay rate of $\alpha_p=\alpha_d=1.1534$. These are similar to the estimate of $\alpha=1.153$ as reported in~\cite{mondie2005exponential}.

\end{ex}

\begin{ex}[Exponential Stability of a Reaction-Diffusion PDE]\label{ex:reaction-diffusion}
Consider the following reaction-diffusion equation.
\begin{align*}
&\dot{\mbf v}(t,s) = 2 \mbf v(t,s) + \partial_s^2\mbf v(t,s), ~~ \mbf v(t,0)=\partial_s \mbf v (t,1)=0.
\end{align*}
The PIE representation of this PDE is defined by parameters
{\small
\begin{align*}
&\mcl T = \fourpi{\emptyset}{\emptyset}{\emptyset}{0,-\theta,-s}, \mcl A = \fourpi{\emptyset}{\emptyset}{\emptyset}{1,-\lambda\theta, -\lambda s}.
\end{align*}} Using Thm.~\ref{thm:stable_LPI}, the maximum provable primal and dual lower bounds on exponential decay rate are $\alpha_p=\alpha_d = 0.4674$. One can find an analytical solution to the above PDE (using the change of variable $\mbf y(t,s) = e^{-2t}\mbf x(t,s)$) and observe that the largest eigenvalue of this solution is $-0.4674$ -- demonstrating accuracy of the maximal lower bounds obtained from the LPIs.
\end{ex}

\subsection{Verifying Eq.~\eqref{eq:dual_io} through numerical simulation}\label{subsec:io_numerical}
The key relation between inputs and outputs of the primal and dual PIE obtained in Eq.~\eqref{eq:dual_io} 
 can be verified numerically. To perform this numerical verification, we simulate various primal and dual PIE systems and measure any error in Eq.~\eqref{eq:dual_io}. For each example, we use zero initial conditions, time interval $t\in [0,5]$, and arbitrarily chosen $L_2$-bounded disturbances $w(t) = \sin(5t)e^{-2t}$ (for the primal) and $\bar{w}(t) = (t-t^2)e^{-t}$ (for the dual). The simulated responses $z$ (primal) and $\bar{z}$ (dual) are then used to measure the error in Eq.~\eqref{eq:dual_io} where error is defined as
\[err(t)= \int_{0}^{t} \bar{z}(\theta)^Tw(t-\theta)d\theta - \int_{0}^{t}\bar{w}(\theta)^Tz(t-\theta)d\theta.\]
The verification is performed for three PIEs obtained from the following three illustrative PDE systems:
\begin{enumerate}
  \item[(E1)] $\dot{\mbf v}(t,s)= \partial_s^2 \mbf v(t,s)+w(t)$, $\mbf v(t,0)=\mbf v(t,1)=0$, $z(t)=\int_0^1 \mbf v(t,s) ds$.
  \item[(E2)] $\dot{\mbf v}(t,s)= -\partial_s \mbf v(t,s)+w(t)$, $\mbf v(t,0)=0$, $z(t)=\mbf v(t,1)$.
  \item[(E3)] $\dot{\mbf v}(t,s)= 3\mbf v(t,s)+\partial_s^2 \mbf v(t,s)+w(t)$, $\mbf v(t,0)=\partial_s \mbf v(t,1)=0$, $z(t)=\mbf v(t,1)$.
\end{enumerate}
For each PDE, the formulae in \cite[Block 4 and 5]{shivakumar_representation_TAC} are used to obtain the parameters $\{\mcl T, \mcl A, \mcl B, \mcl C, \mcl D\}$. The results are given in Table \ref{tab:IO_compare} and indicate almost no numerical error over the given time interval for all three examples.

\begin{table}[t]
\begin{center}
\begin{tabular}{|c|c|c|c|}  
\hline
Example&(E1)&(E2)&(E3) \\
\hline
$err(t=1.0)$&1.2e-07&-3.2e-06&2.1e-06\\\hline
$err(t=2.5)$&-4.7e-07&2.8e-05&-2.3e-05\\\hline
$err(t=5.0)$&-2.2e-07&1.6e-05&-2.3e-04\\\hline
\end{tabular}
\end{center}
\caption{The error in Eq.~\eqref{eq:dual_io} at 3 time instances for numerical simulation of the primal and dual PIEs obtained from Examples E1, E2, and E3; subject to zero initial conditions and disturbances $w=\sin(5t)\exp(-2t)$ and $\bar{w}=(t-t^2)\exp(-t)$; and where error is defined as $err(t)= \int_{0}^{t} \bar{z}(\theta)^Tw(t-\theta)d\theta - \int_{0}^{t}\bar{w}(\theta)^Tz(t-\theta)d\theta$.
}\label{tab:IO_compare}
\end{table}

Note that Example (E3) is unstable and hence its primal and dual PIE representation are likewise unstable. However, as mentioned in Remark~\ref{rem:intertwining}, the intertwining relationship in Eq.~\eqref{eq:dual_io} does not require stability -- an assertion verified by the numerical analysis in Table~\ref{tab:IO_compare}.

\subsection{Comparison of $L_2$-gain Bounds from Thm.~\ref{thm:gain_lpi} and Thm.~\ref{thm:gain_lpi_new}}\label{subsec:numeric_L2gain}
Next, we examine the practical impact on accuracy of the class of extended storage functions described in Subsections~\ref{subsec:alternativeLPI} and~\ref{subsec:alternativeLPI_control}. For this analysis, we compute the minimum $L_2$-gain bounds for several PDEs using both the primal and dual LPI conditions in Thm.~\ref{thm:gain_lpi} and the extended primal and dual LPI conditions in Thm.~\ref{thm:gain_lpi_new}. The PDEs included in this test are: heat eq. with Dirichlet BC's (A.1); wave equation (A.2); heat eq. with mixed BC's (A.3); coupled heat/transfer eq. (A.4); and coupled heat eqs. (A.5). 
\begin{align*}
    A.1\quad&\dot{\mbf v}(t,s) = \partial_s^2\mbf v(t,s) + w(t), \quad \mbf v(t,0)=\mbf v(t,1)=0,\\ 
    &z(t) = \int_0^1 \mbf v(t,s) ds.\\
    A.2\quad &\ddot{\mbf v}(t,s) = \partial_s^2\mbf v(t,s) + w(t), \quad \mbf v(t,0)=\mbf v(t,1)=0,\\ 
    &z(t) = \mbf v_s(t,1).\\
    A.3\quad &\dot{\mbf v}(t,s) = \partial_s^2\mbf v(t,s) + w(t), \; \mbf v(t,0)=\mbf v_s(t,1)=0,\\
    & z(t) = \int_0^1 \mbf v(t,s) ds.\\
    A.4\quad&\bmat{\dot{\mbf v}_1(t,s)\\\dot{\mbf v}_2(t,s)} = \partial_s \bmat{\mbf v_2(t,s)\\\mbf v_1(t,s) }+\bmat{0\\w(t)},\; \mbf v_2(t,0) =0,\\& \mbf v_1(t,1) + k\mbf v_2(t,1) = 0,\quad z(t) = \int_0^1 \mbf v_2(t,s) ds.\\
    A.5\quad &\dot{\mbf v}(t,s) = \bmat{1& 1.5\\ 5&0.2}\mbf v(t,s)+\frac{1}{2.6}\partial_s^2 \mbf v(t,s) + sw(t), \\
    &\mbf v(t,0)=\mbf v(t,1)=0, \quad z(t) = \int_0^1 \bmat{1&0}\mbf v(t,s) ds.
\end{align*}
The results are listed in Table~\ref{tab:one}. Accuracy of the $L_2$-gain bound from Thms.~\ref{thm:gain_lpi} and~\ref{thm:gain_lpi_new} can be inferred from the gap between primal (Thms.~\ref{thm:gain_lpi}a and~\ref{thm:gain_lpi_new}a) and dual LPIs (Thms.~\ref{thm:gain_lpi}b and~\ref{thm:gain_lpi_new}b). These results indicate that both the primal and dual LPIs from Thm.~\ref{thm:gain_lpi_new} are highly accurate. However, it seems that the restriction on the structure of the storage functional in the dual test of Thm.~\ref{thm:gain_lpi} results in conservatism in several cases. Furthermore, in Example A.2, the primal test of Thm.~\ref{thm:gain_lpi} is likewise conservative. These results indicate the importance of using the extended class of storage functionals described in Subsections~\ref{subsec:alternativeLPI} and~\ref{subsec:alternativeLPI_control}.
 
\begin{table}[!ht]
\begin{center}
\begin{tabular}{ |c|l|l||l|l| } \hline
\multicolumn{5}{|c|}{Bound on $\hinf$-norm} \\\hline
  & Thm.~\ref{thm:gain_lpi}a & Thm.~\ref{thm:gain_lpi}b& Thm.~\ref{thm:gain_lpi_new}a& Thm.~\ref{thm:gain_lpi_new}b\\ \hline
 Ex.~A.1&0.083&6.89&0.0833&0.0833\\ \hline
 Ex.~A.2&4.23&8.71&0.5& 0.5 \\ \hline
 Ex.~A.3&0.33&3.49&0.33&0.33\\ \hline
 Ex.~A.4&0.8418&0.8418&0.842&0.8418 \\ \hline
 Ex.~A.5& 0.81&3.501 &0.81&0.81\\ \hline
\end{tabular}
\caption{Minimum computed bounds on the $L_2$-gain for PDEs A.1-A.5 using the primal (a) and dual (b) LPIs in Theorems \ref{thm:gain_lpi} and \ref{thm:gain_lpi_new}. A gap between primal (a) and dual (b) bounds indicates conservatism in one of the computed bounds. The absence of any gaps between 12a and 12b suggests that the bounds on $H_\infty$ norm computed using Thm.~\ref{thm:gain_lpi_new} are not conservative.
}
\label{tab:one}
\end{center}
\end{table}

\subsection{$\hinf$-Optimal State-Feedback Control using Cor.~\ref{cor:optimal_control_lpi}}\label{subsec:num_control}
We conclude this section by synthesizing optimal state-feedback controllers for three canonical examples of PDE control: 1) the Euler-Bernoulli beam equation with in-domain actuation; 2) a reaction-diffusion PDE with actuation at the boundary; and 3) the wave equation with actuation at the boundary. In each case, motivated by the numerical tests in Subsec.~\ref{subsec:numeric_L2gain}, we use the extended class of storage functions for which LPI conditions are given in Cor.~\ref{cor:optimal_control_lpi}.
Table~\ref{tab:acc_compare} summarizes the results for all three examples, providing the minimum achievable closed-loop $L_2$-gain ($H_\infty$-norm) and computation time as function of the degree ($n=1,2,3,4$) of the monomial bases as defined in Eq.~\eqref{eq:monomial_bases}. Computation time is defined as time required to set up the LPI, solve the LPI, and reconstruct the controller gains using a desktop computer with Intel Core \textit{i7-5960X} CPU and 64GB DDR4 RAM. 

For each example, we specify the corresponding PIE state and use PIETOOLS to obtain the associated PIE representation. The optimal controller gains (as calculated for order $n=3$) are obtained using the procedure defined in Sec.~\ref{sec:controller}. For Examples \ref{ex:hinf_diffusion} and \ref{ex:hinf_beam}, numerical simulation of the closed-loop response is obtained using PIESIM~\cite{Peet_2024ACM} to verify that the closed-loop $L_2$-gain bound is satisfied. 

%

\begin{ex}[Euler-Bernoulli beam equation]\label{ex:hinf_beam}
Recall that in Example~\ref{ex:EB-representation}, we formulated the problem of optimal control of an Euler-Bernoulli (EB) beam model as using the PIE state $\mbf x=\partial_s^2\mbf v$.
Solving the LPI in Cor.~\ref{cor:optimal_control_lpi}, we find the $H_\infty$-optimized state-feedback controller to be 
{
\begin{align*}
&u(t)=\int_0^1\bmat{Q_{a}(s)&Q_{b}(s)}\partial_s^2\mbf v(s) ds,\\
&Q_{a}(s) \hspace{-0.5mm}=\hspace{-0.5mm} 0.23s^5 - 1.37s^4 + 1.9s^3 - 2.1s^2 + 2.03s - 0.87,\notag\\
&Q_{b} (s) \hspace{-0.5mm}=\hspace{-0.5mm} 
0.02s^5 - 0.04s^4 + 0.06s^3 - 0.11s^2 - 0.01s + .0006.\notag
\end{align*}
}
The upper bound on the $\hinf$-norm of the corresponding closed-loop PDE obtained from the LPI in Cor.~\ref{cor:optimal_control_lpi} is $0.73$. In Figures \ref{fig:optimal_control_ex22}a, \ref{fig:optimal_control_ex22}b and \ref{fig:optimal_control_ex2_out}, we plot the system response for a disturbance $w(t) = \sin(3t)e^{-t}$ with the zero initial conditions. The $L_2$-gain for this disturbance is $0.099$ which is less than the predicted worst-case bound of $.73$. 

\end{ex}

\begin{figure}[t]
\centering
\begin{subfigure}[b]{0.24\textwidth}
\centering
\includegraphics[width=\textwidth]{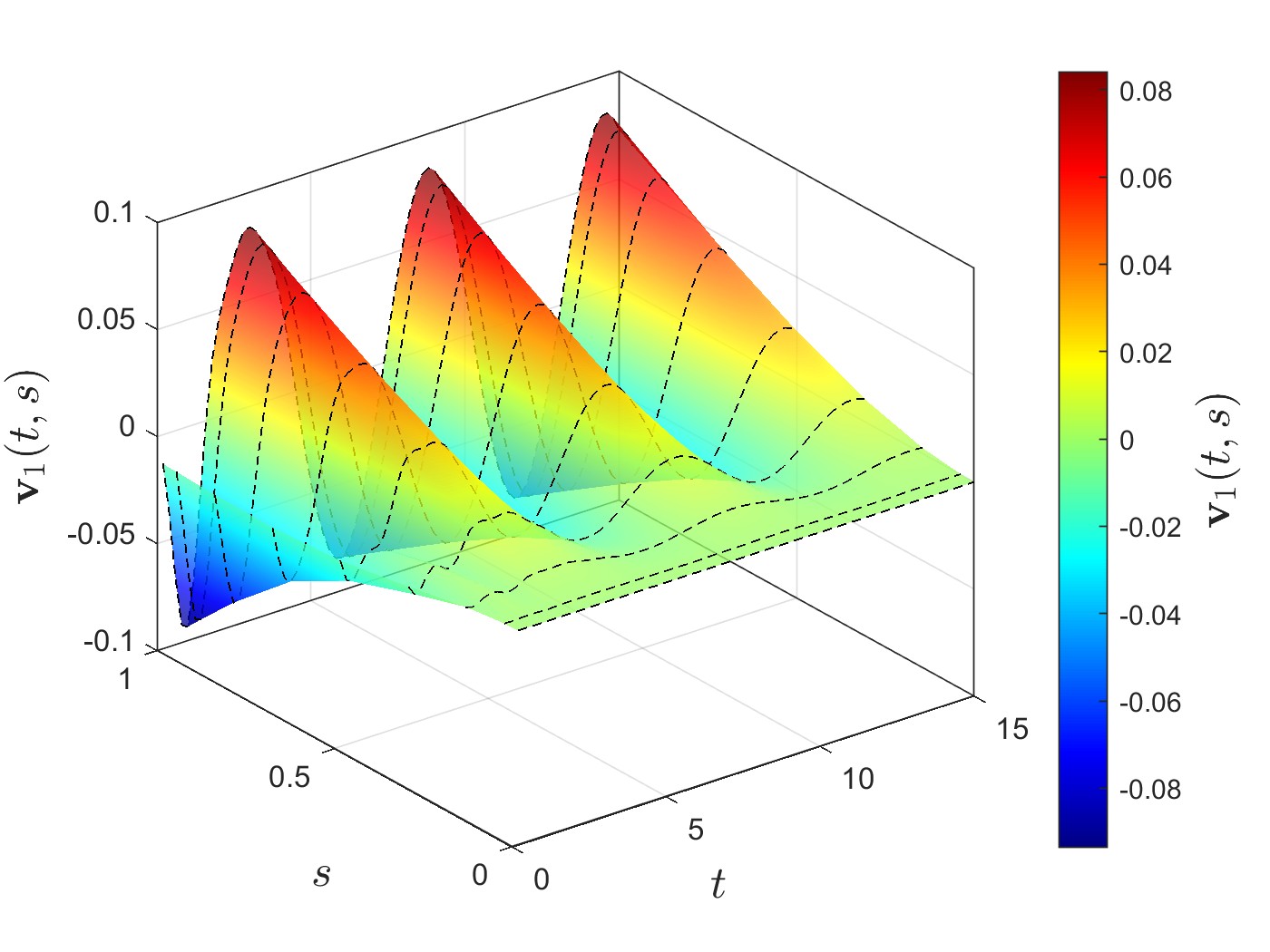}
\caption{Open loop response}\label{fig:optimal_control_ex2}
\end{subfigure}
\begin{subfigure}[b]{0.24\textwidth}
\centering
\includegraphics[width=\textwidth]{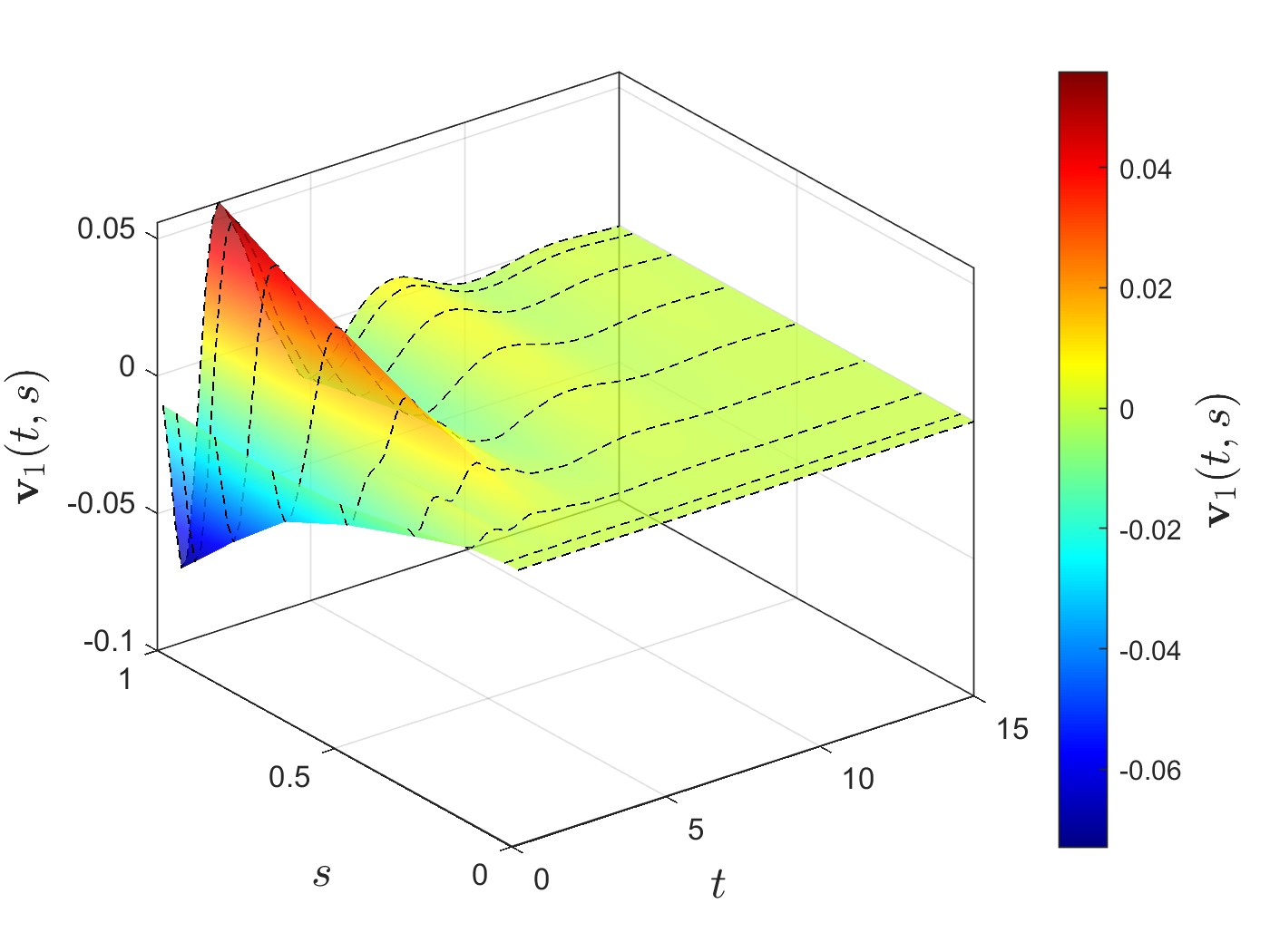}
\caption{Closed loop response}\label{fig:optimal_control_ex2b}
\end{subfigure}
\caption{
Numerical simulation of open loop (a) and closed-loop (b) response of $\mbf v_1(t,s)$ for in-domain control of an Euler-Bernoulli beam equation (Ex.~\ref{ex:hinf_beam}) with disturbance $w(t) = \sin(3t)e^{-t}$. As seen in (a), when $u=0$, the system oscillates indefinitely. With feedback control, however, the state demonstrates exponential decay, as seen in (b).
%
}\label{fig:optimal_control_ex22}
\end{figure}

\begin{figure}[t]
\centering
\includegraphics[width=\linewidth]{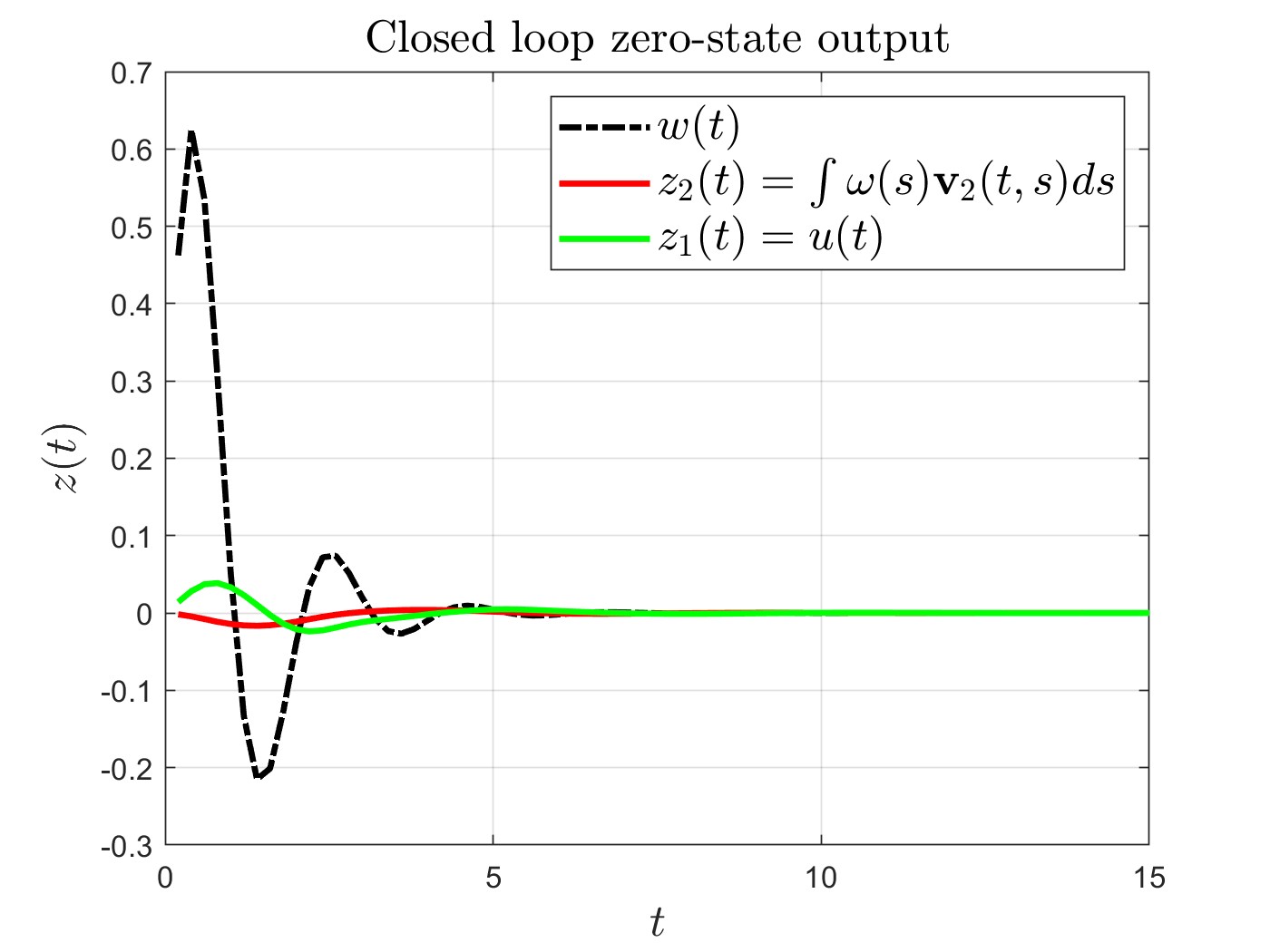}
\captionof{figure}{
Closed-loop numerical simulation of regulated outputs $z_1=x(t)$ and $z_2(t)=\int_0^1\omega(s)\mbf v_2(t,s) ds$ with disturbance $w(t) = \sin(3t)e^{-t}$ for Ex.~\ref{ex:hinf_beam} where $\omega(s):=(1-s)^2/2$. }\label{fig:optimal_control_ex2_out}
\end{figure}

\begin{ex}[Reaction-Diffusion Equation] \label{ex:hinf_diffusion}
Consider boundary control of an unstable reaction-diffusion equation:
{
\begin{align*}
&\dot{\mbf v}(t,s) = 5\mbf v(t,s) +\partial_s^2\mbf v(t,s)+w(t),~\dot{x}(t) =u(t), \\
&z(t) = \bmat{x(t)\\\int_0^1 \mbf v(t,s)ds},~\mbf v(t,0) =0,~\partial_s\mbf v(t,1) = x(t),
\end{align*}}
The corresponding PIE has state $\mbf x(t)=\partial_s^2\mbf v(t)$ and parameters
{\small
\begin{align*}
&\mcl T = \fourpi{1}{0}{s}{0,-\theta,-s}, \mcl A=I + 5\mcl T, \mcl B_1 = \fourpi{0}{\emptyset}{1}{\emptyset}, \\
&\mcl B_2 = \fourpi{1}{\emptyset}{0}{\emptyset}, \mcl C= \fourpi{\bmat{1\\0.5}}{\bmat{0\\0.5s^2-s}}{\emptyset}{\emptyset}.
\end{align*}}
The $H_\infty$-optimized state-feedback controller is
\begin{align*}
&u(t)=-6.71x(t)+10^3 \int_0^1 K(s)\partial_s^2\mbf v(s) ds,\\
&K(s) = -11.68s^8 + 44.23s^7 - 65.93s^6 + 49.38s^5- 19.82s^4\\
&\qquad\qquad + 4.27s^3 - 0.46s^2 + 0.02s - 0.0002.
\end{align*}
The upper bound on the $\hinf$-norm of the corresponding closed-loop PDE obtained from the LPI in Cor.~\ref{cor:optimal_control_lpi} is $4.59$. The simulated $L_2$-gain under disturbance $w(t) = \sin(5t)e^{-t}$ is $1.8905$ which verifies the bound.
In Figures \ref{fig:optimal_control_ex12}a, \ref{fig:optimal_control_ex12}b and \ref{fig:optimal_control_ex1_out}, we plot the system response for a disturbance $w(t) = \sin(5t)e^{-t}$.
\end{ex}

\begin{figure}[t]
\centering
\begin{subfigure}[b]{0.24\textwidth}
\includegraphics[width=\textwidth]{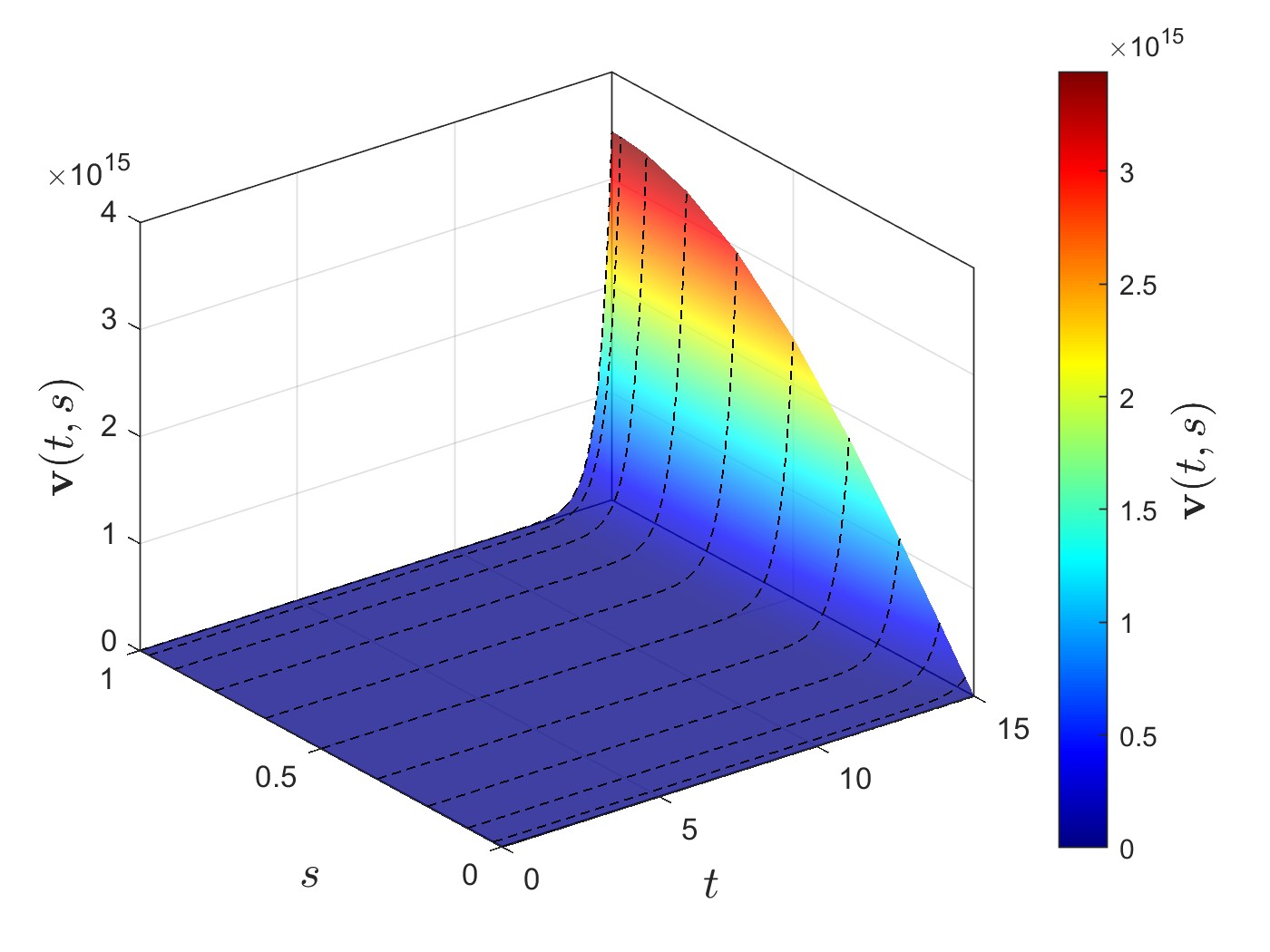}
\caption{Open loop response}\label{fig:optimal_control_ex1}
\end{subfigure}
\begin{subfigure}[b]{0.24\textwidth}
\includegraphics[width=\textwidth]{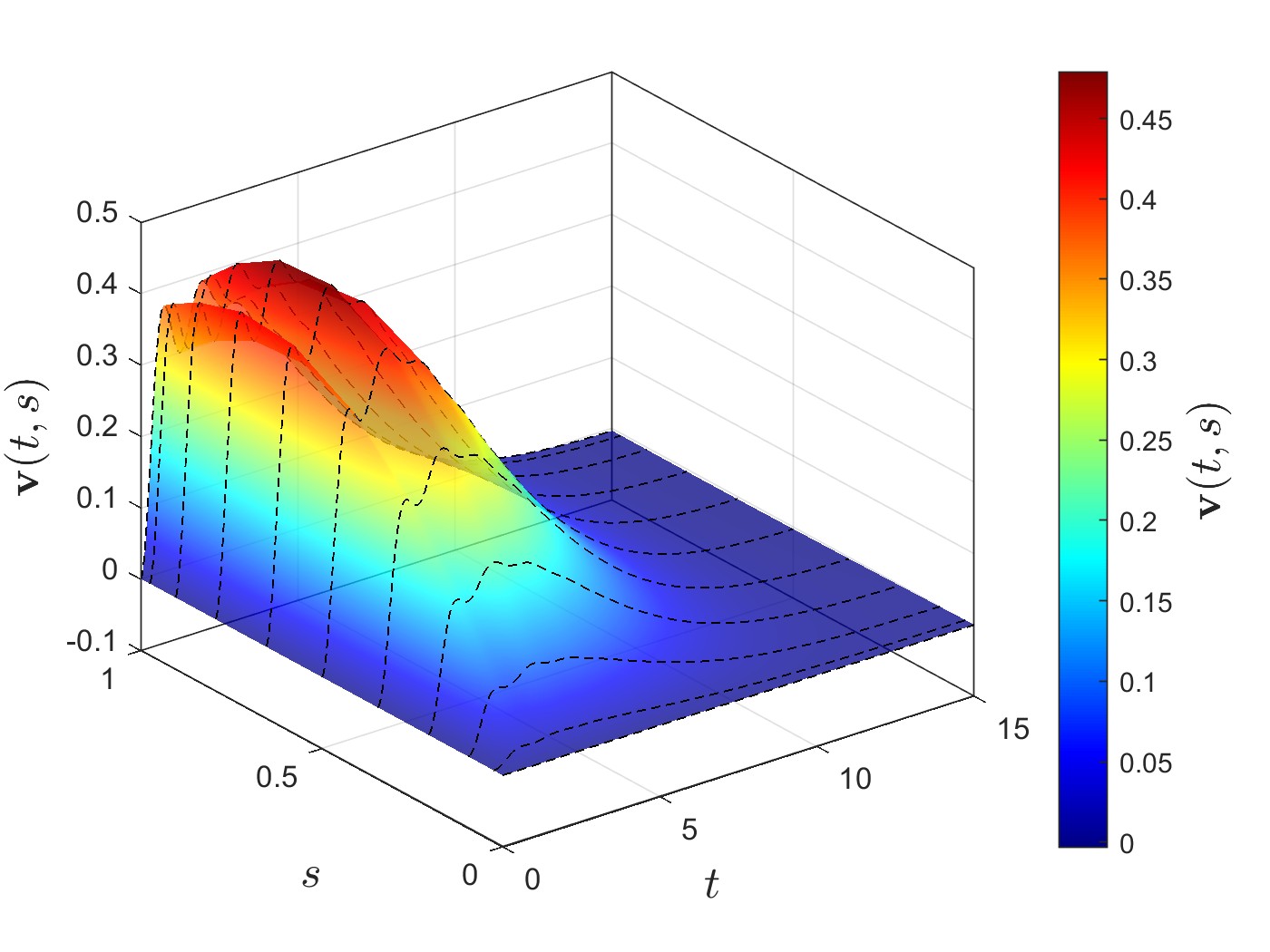}
\caption{Closed loop response}\label{fig:optimal_control_ex1b}
\end{subfigure}
\caption{
Numerical simulation of open loop (a) and closed-loop response (b) for boundary control of the reaction-diffusion equation (Ex.~\ref{ex:hinf_diffusion}) with disturbance $w(t) = \sin(5t)e^{-t}$. As expected, the uncontrolled system is unstable as seen in (a). However, with feedback control, the closed-loop system is stable as seen in (b).
}\label{fig:optimal_control_ex12}
\end{figure}

\begin{figure}[t]
\centering
\includegraphics[width=\linewidth]{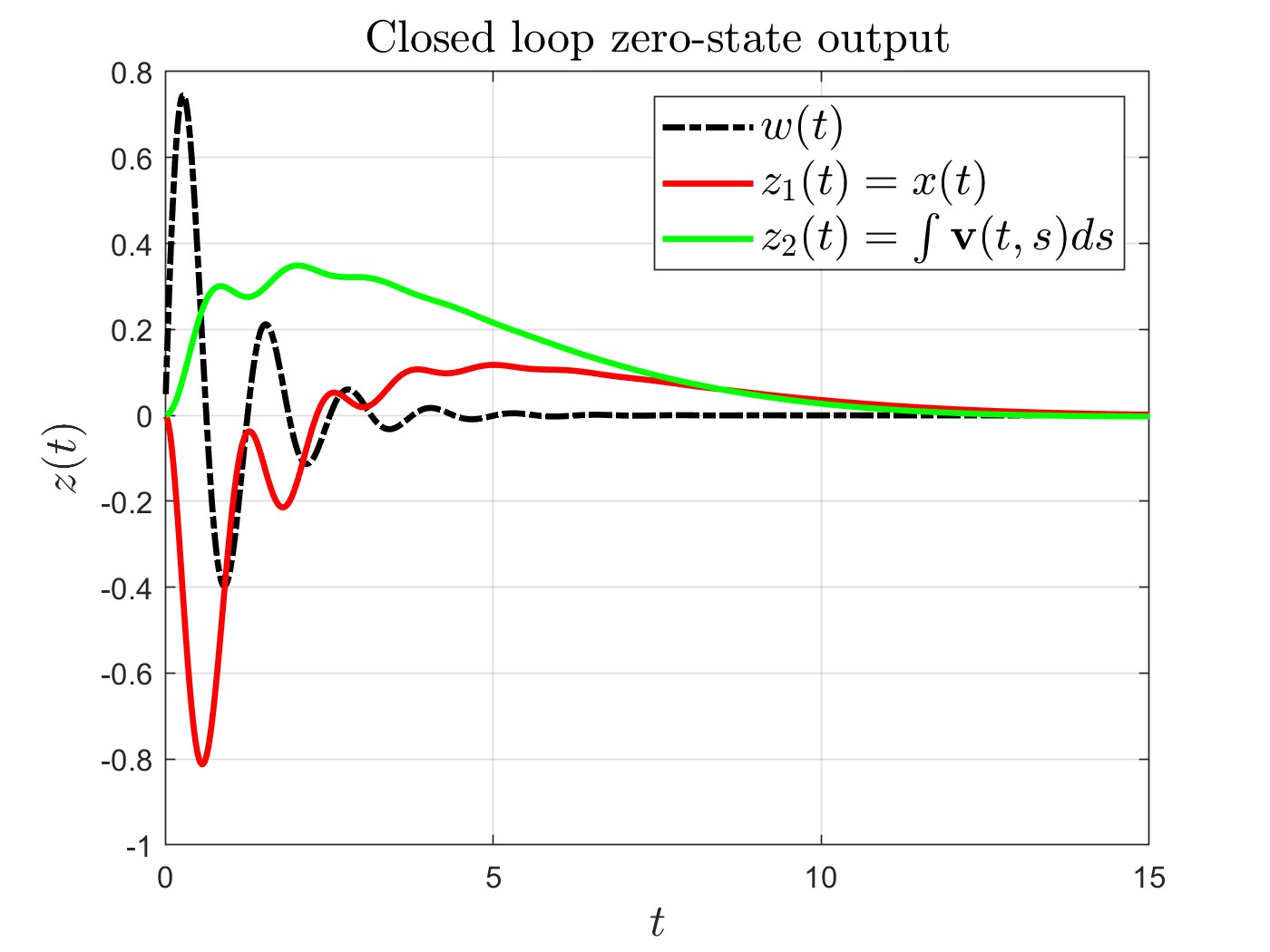}
\captionof{figure}{Closed-loop numerical simulation of regulated outputs $z_1(t)=x(t)$ and $z_2(t)=\int_0^1 \mbf v(t,s)ds$ with disturbance $w(t) = \sin(5t)e^{-t}$ for Ex. \ref{ex:hinf_diffusion}.}\label{fig:optimal_control_ex1_out}
\end{figure}

\begin{ex}[Wave equation]\label{ex:hinf_wave}
Consider boundary control of a wave equation:
\begin{align*}
&\ddot{\mbf \eta}(t,s) = \partial_s^2\mbf \eta(t,s)+w(t),~\dot{x}(t) = u(t)\notag\\
&z(t) = \bmat{x(t)\\\int_0^1 \mbf \eta(t,s) ds},~\mbf \eta(t,0) = 0, ~ \partial_s \mbf \eta(t,1) =x(t).
\end{align*}
To eliminate the second-order time derivative, $\ddot \eta$, we define $\mbf{v} = \bmat{\mbf \eta & \dot{\mbf \eta}}^T$ to obtain
\begin{align*}
&\dot{\mbf{v}}(t,s) \hspace{-1mm}=\hspace{-1mm} \bmat{0&\hspace{-2mm}1\\0&\hspace{-2mm}0}\mbf{v}(t,s)\hspace{-1mm}+\hspace{-1mm}\bmat{0&\hspace{-2mm}0\\1&\hspace{-2mm}0}\partial_s^2\mbf{v}(t,s)\hspace{-1mm}+\hspace{-1mm}\bmat{0\\1}w(t),~\dot{x}(t)\hspace{-1mm}=\hspace{-1mm}u(t), \\
&z(t) \hspace{-1mm}=\hspace{-1mm} \bmat{v(t)\\\int_0^1 \bmat{1&\hspace{-2mm}0}\mbf v(t,s) ds},~\bmat{1&\hspace{-2mm}0&\hspace{-2mm}0&\hspace{-2mm}0\\0&\hspace{-2mm}0&\hspace{-2mm}1&\hspace{-2mm}0}\bmat{\mbf{v}(t,0)\\\mbf{v}(t,1)} \hspace{-1mm}=\hspace{-1mm}\bmat{0\\1}x(t).
\end{align*}
We convert to a PIE which has state 
\[
\mbf x(t)=\bmat{x(t)&\partial^2_s \mbf v_{1}(t)&\mbf v_2(t)}^T=\bmat{x(t)&\partial_s^2\eta(t)&\dot \eta (t)}^T
\]
The upper bound on the $\hinf$-norm of the corresponding closed-loop PDE obtained from the LPI in Cor.~\ref{cor:optimal_control_lpi} is $.64$. 
The $H_\infty$-optimized state-feedback controller is
\[
u(t) = -0.17x(t)+10^{-2}\int_0^1 Q_{1}(s) \partial_s^2\eta(t,s)+
Q_{2}(s)\dot \eta(t,s)ds
\]
where\vspace{-2mm}

{\small
\begin{align*}
&{ Q_{1}(s) = .5 s^8 - 2s^7 + 3 s^6 - 2s^5 - 30s^4 + 60s^3 - 70s^2 + 20s - .8,}\notag\\
&Q_{2} (s) = .2 s^8 - .7 s^7  + .6 s^5 - 5s^4 - 20s^3 + 80s^2 - 2s - 40.\notag
\end{align*}}
\end{ex}

\begin{table}
\begin{center}
\begin{tabular}{|c|l|l|l|l|}\hline
\multicolumn{5}{|c|}{Closed-loop $\hinf$-norm vs. monomial degree} \\
\hline
Degree, n &1&2&3&4 \\\hline
Ex.~\ref{ex:hinf_beam}&3.29 & 0.89 &0.73 &0.66 \\\hline
Ex.~\ref{ex:hinf_diffusion}&7.86 &5.11 &4.59 &4.25 \\\hline
Ex.~\ref{ex:hinf_wave}&0.65 &0.64 &0.64 &0.639 \\\hline
\end{tabular}
\end{center}
\caption{Achieved bound on $\hinf$-norm for the closed-loop PDE with state-feedback for Examples~\ref{ex:hinf_beam},~\ref{ex:hinf_diffusion}, and~\ref{ex:hinf_wave} in Sec.~\ref{subsec:num_control}. Closed-loop norm is the optimal value of $\gamma$ in Cor.~\ref{cor:optimal_control_lpi} wherein variables $\mcl P$ and $\mcl Z$ are parameterized for degree $n$ as $\mcl P = \mcl Z_n^*Q_p\mcl Z_n$, $\mcl Z = Q_z\mcl Z_n$, and $Q_p\ge 0$, $Q_z$ are matrices and where $\mcl Z_n$ is defined in Eq.~\eqref{eq:monomial_bases}.}\label{tab:acc_compare}
\end{table}

\begin{table}
\begin{center}
\begin{tabular}{|c|l|l|l|l|}\hline
\multicolumn{5}{|c|}{Computation time vs. monomial degree} \\
\hline
Degree, n &1&2&3&4 \\\hline
Ex.~\ref{ex:hinf_beam}&6.2&15&30.5 &67.4\\\hline
Ex.~\ref{ex:hinf_diffusion}&4.6&5.5&9.5&13.9\\\hline
Ex.~\ref{ex:hinf_wave}&10.9&27.1&49.5&87.5\\\hline
\end{tabular}
\end{center}
\caption{Required computation time to find state-feedback controllers for Examples~\ref{ex:hinf_beam},~\ref{ex:hinf_diffusion}, and~\ref{ex:hinf_wave} in Sec.~\ref{subsec:num_control}. As in Table \ref{tab:acc_compare}, the norm is found by minimizing $\gamma$ in Cor.~\ref{cor:optimal_control_lpi}. The values correspond to the total CPU runtime, in seconds, for solving the $\hinf$-optimal state-feedback problem --- i.e., time for setting up the LPIs, solving the LPIs, and controller reconstruction.}\label{tab:acc_compare_2}
\end{table}

From the estimates in Table~\ref{tab:acc_compare}, we conclude that increasing the monomial degree ($n$) may (e.g. for diffusion PDEs) or may not (e.g. for transport PDEs) be needed to achieve near-optimal performance of the closed-loop system. The computational costs associated with increasing the degree are evaluated in  Tab. \ref{tab:acc_compare_2}.

\begin{ex}[Heat Equation]\label{ex:heat}
    For the heat equation as formulated in this example, there exists an analytic construction of the $H_\infty$-optimal controller. For this reason, we use this simple example as a benchmark, demonstrating convergence of the synthesized closed loop $H_\infty$-gains obtained from Corollary \ref{cor:optimal_control_lpi} to this analytic limit. Specifically, the PDE is defined as
    \begin{align*}
        \dot{\mbf x}(t,s) &= \partial_s^2 \mbf x(t,s) +w(t) +u(t)\\
        z(t) &= \bmat{u(t)\\\int_0^1 \mbf x(t,s) ds},\quad \mbf x(t,0)=\mbf x(t,1) = 0.
    \end{align*}
In \cite{bergeling2020closed}, it was shown that for this formulation, $ u(t) = \int_0^1 \frac{1}{2}s(s-1)\mbf x(t,s) ds$ provides an $\hinf$-optimal state-feedback controller with the associated $\hinf$-norm of $0.0164$ for the closed-loop system. For comparison, the closed-loop gain obtained from Cor. \ref{cor:optimal_control_lpi} with degree $n=4$ is $0.0166$.
\end{ex}
\section{Conclusions}
Recent work has shown that a large class of linear PDE systems admit an equivalent state-space representation using Partial Integral Equations (PIEs). However, relatively little is known about the system properties of such PIEs. In this paper, we have presented a series of duality results which establish equivalence in the input-output properties of a PIE and its dual. Specifically, for any given PIE, we have shown how to construct a dual PIE with simple formulae for obtaining the system parameters of this dual. Then, by establishing an intertwining property between the solutions of the dual and primal PIE, we have shown that the primal and dual PIEs have at least three equivalent properties: asymptotic stability, exponential decay rate, and $L_2$-gain. 

A benefit of the PIE representation of PDE systems is the ability to generalize Linear Matrix Inequality (LMI) conditions to Linear PI Inequality (LPI) conditions which may then be solved using convex optimization. Taking this approach, we show that the LPIs for stability, exponential decay rate, and $L_2$-gain admit equivalent primal and dual formulations. The dual formulations of the stability and $L_2$-gain LPIs are then used to solve the problem of stabilizing and $\hinf$-optimal state-feedback controller synthesis. A numerical algorithm is then proposed for reconstructing controller gains and implementation of the controller. Finally, numerical testing is used to verify the theorems and obtain controllers with provable $\hinf$-norm bounds. The numerical results show no apparent sub-optimality in the resulting controllers or $\hinf$ bounds.

\blue{Finally, we note that while the duality results and synthesis conditions presented here do not encompass performance metrics such as state-to-output gain, such extensions may be possible using the proposed methodology.}

\bibliographystyle{ieeetr}
\bibliography{optimal_control,peet_bib}
\vspace{-10 mm}	
\begin{IEEEbiography}[\vspace*{-6mm}{\includegraphics[width=1in,height=1.15in,clip,keepaspectratio]{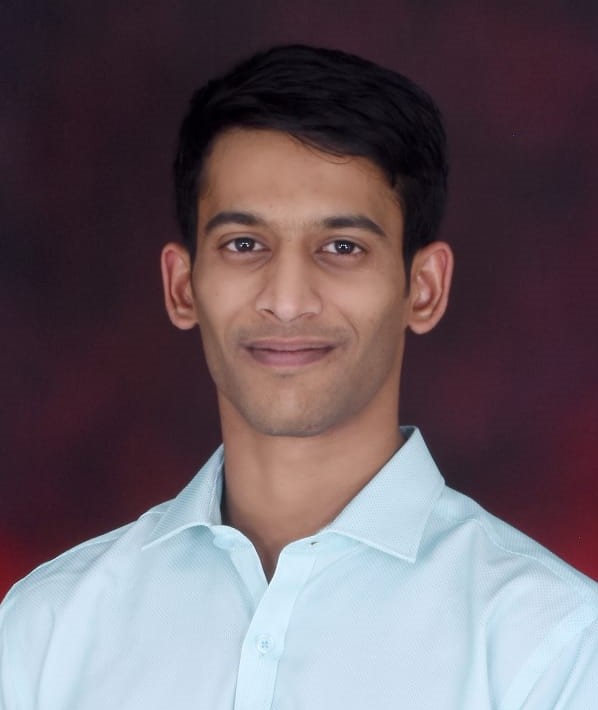}}]{Sachin Shivakumar} received his B.Tech. (2015) in Mechanical Engineering from the Indian Institute of Technology - Kharagpur, and both M.S. and Ph. D. in Mechanical Engineering from Arizona State University. In 2024, he was a part of Department of Aerospace Engineering of Iowa State University working as a Postdoctoral Researcher. Currently, he is a Postdoctoral Researcher at Applied Math \& Plasma Physics Division of Los Alamos National Laboratory.
\end{IEEEbiography}
\vspace{-10 mm}
\begin{IEEEbiography}[\vspace*{-6mm}{\includegraphics[width=1in,height=1.15in,clip,keepaspectratio]{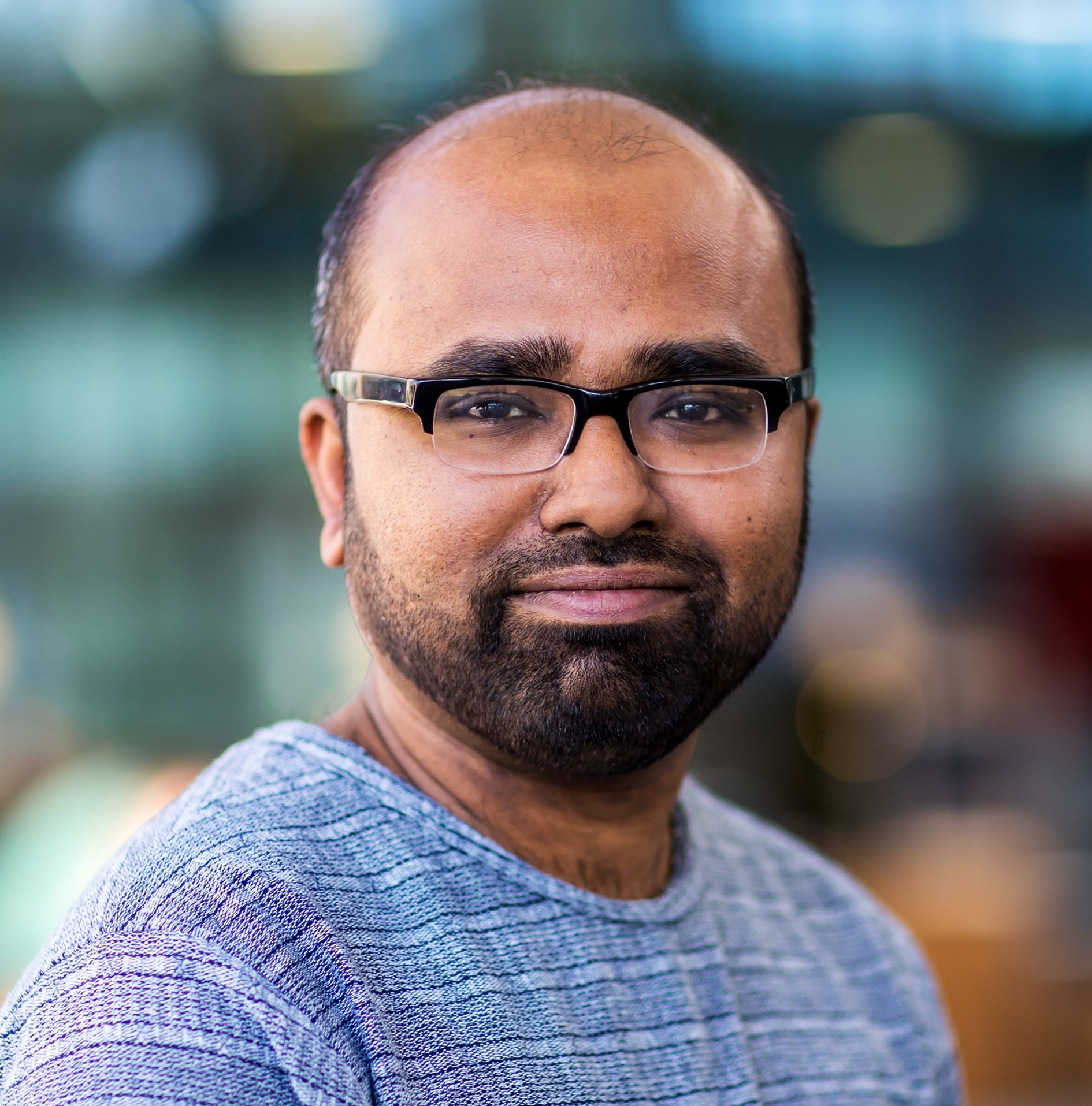}}]{Amritam Das} completed his M.S. and Ph.D. at Eindhoven University of Technology in 2020. He was a research associate at the University of Cambridge, affiliated with Sidney Sussex college, and a post-doctoral scholar at KTH Royal Institute of Technology. His is currently an Assistant Professor in the Control Systems group in the Department of Electrical Engineering of TU Eindhoven.

\end{IEEEbiography}	
\vspace{-10 mm}
\begin{IEEEbiography}[\vspace*{-6mm}{\includegraphics[width=1in,height=1.15in,clip,keepaspectratio]{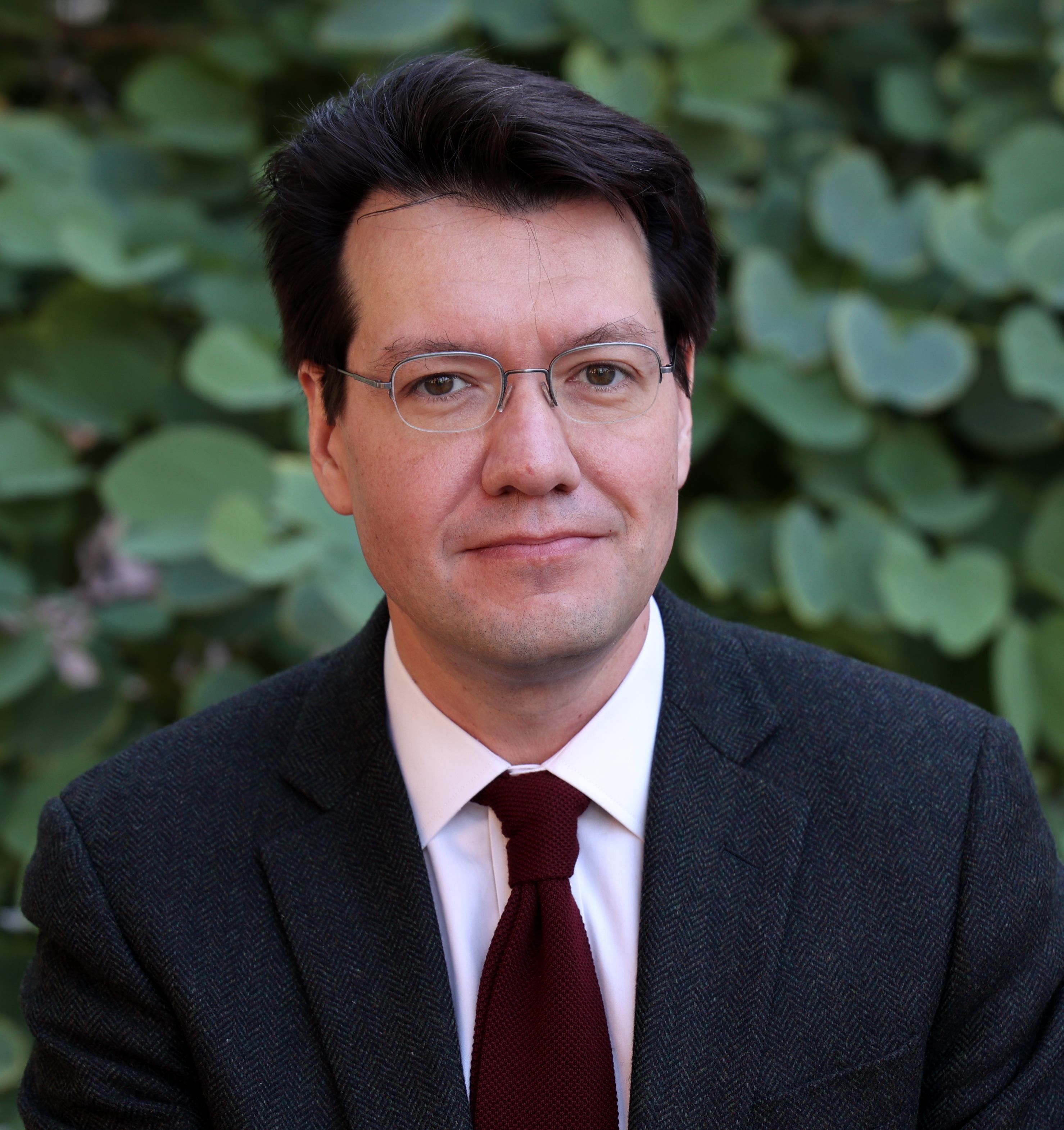}}]{Matthew Peet}
received M.S. and Ph.D. degrees in aeronautics and astronautics
from Stanford University (2000-2006). He was a postdoc at INRIA (2006-2008) and Asst. Professor at the Illinois Institute of Technology (2008-2012). Currently, he
is an Associate Professor of Aerospace Engineering
at Arizona State University.
\end{IEEEbiography}
\vfill

\end{document}